\documentclass[a4paper]{article}

\usepackage{amsmath,amsfonts,amsthm,amssymb,graphicx,tikz-cd,stmaryrd}
\usepackage{hyperref}
\usepackage{aliascnt}
\usepackage{cleveref}

\theoremstyle{definition}
\newtheorem{theorem}{Theorem}[section]
\newaliascnt{definition}{theorem}
\newtheorem{definition}[definition]{Definition}
\aliascntresetthe{definition}
\newaliascnt{lemma}{theorem}
\newtheorem{lemma}[lemma]{Lemma}
\aliascntresetthe{lemma}
\newaliascnt{proposition}{theorem}
\newtheorem{proposition}[proposition]{Proposition}
\aliascntresetthe{proposition}
\newaliascnt{corollary}{theorem}
\newtheorem{corollary}[corollary]{Corollary}
\aliascntresetthe{corollary}
\newaliascnt{remark}{theorem}
\newtheorem{remark}[remark]{Remark}
\aliascntresetthe{remark}
\newaliascnt{example}{theorem}
\newtheorem{example}[example]{Example}
\aliascntresetthe{example}
\newaliascnt{problem}{theorem}

\aliascntresetthe{problem}

\crefname{lemma}{Lemma}{Lemmas}
\crefname{definition}{Definition}{Definitions}
\crefname{theorem}{Theorem}{Theorems}
\crefname{proposition}{Proposition}{Propositions}
\crefname{corollary}{Corollary}{Corollaries}
\crefname{remark}{Remark}{Remarks}
\crefname{example}{Example}{Examples}
\crefname{problem}{Problem}{Problems}

\Crefname{lemma}{Lemma}{Lemmas}
\Crefname{definition}{Definition}{Definitions}
\Crefname{theorem}{Theorem}{Theorems}
\Crefname{proposition}{Proposition}{Propositions}
\Crefname{corollary}{Corollary}{Corollaries}
\Crefname{remark}{Remark}{Remarks}
\Crefname{example}{Example}{Examples}
\Crefname{problem}{Problem}{Problems}

\crefalias{enumi}{cond}
\crefformat{cond}{#2 #1 #3}
\Crefformat{cond}{#2 #1 #3}

\crefformat{equation}{#2(#1)#3}
\Crefformat{equation}{#2(#1)#3}

\newcommand{\auth}{ \author{Rin Gotou \thanks{\texttt{u661233h@alumni.osaka-u.ac.jp}}}}
\newcommand{\zahl}{\mathbb{Z}}
\newcommand{\real}{\mathbb{R}}
\newcommand{\cpx}{\mathbb{C}}
\newcommand{\id}{\fct{id}}
\newcommand{\xto}[1]{\xrightarrow{#1}}
\newcommand{\oring}{\mathcal{O}}
\newcommand{\pone}{\mathbb{P}^1}
\newcommand{\proj}{\mathbb{P}}
\newcommand{\af}{\mathbb{A}}
\newcommand{\inc}{\hookrightarrow}
\newcommand{\ratmap}{\dashrightarrow}
\newcommand{\ol}[1]{\overline{#1}}
\newcommand{\fct}[1]{\operatorname{#1}}
\newcommand{\cat}[1]{\mathcal{#1}}
\newcommand{\shf}[1]{\mathcal{#1}}
\newcommand{\oshf}{\shf{O}}
\newcommand{\maxid}{\mathfrak{m}}
\DeclareMathOperator{\spec}{Spec}
\DeclareMathOperator{\mrat}{Rat}
\newcommand{\crat}{\ol{\fct{Rat}}}
\newcommand{\pmcrat}{\ol{\fct{M}_\pm \fct{Rat}}}
\newcommand{\glt}{\fct{GL}_2}
\newcommand{\slt}{\fct{SL}_2}
\newcommand{\pglt}{\fct{PGL}_2}
\newcommand{\gmult}{\mathbb{G}_m}
\newcommand{\eqn}[1]{\[ #1 \]}
\newcommand{\eqnl}[1]{\begin{equation} #1 \end{equation}}
\newcommand{\eqns}[1]{\begin{align*} #1 \end{align*}}
\newcommand{\eqnsl}[1]{\begin{align} #1 \end{align}}
\newcommand{\eqnslg}[2]{\begin{equation} #2 \begin{aligned} #1 \end{aligned} \end{equation}}

\DeclareMathOperator{\ord}{ord}
\DeclareMathOperator{\chow}{Chow}
\DeclareMathOperator{\supp}{supp}
\DeclareMathOperator{\red}{red}
\DeclareMathOperator{\diff}{diff}
\DeclareMathOperator{\spe}{sp}
\DeclareMathOperator{\val}{val}
\DeclareMathOperator{\sgn}{sgn}
\DeclareMathOperator{\chara}{char}
\DeclareMathOperator{\stab}{stab}
\DeclareMathOperator{\depth}{depth}
\DeclareMathOperator{\conf}{Conf}
\DeclareMathOperator{\cc}{CC}
\newcommand{\ponean}{\proj^{1,an}}
\newcommand{\fsfld}{ K }
\newcommand{\fib}[1]{\mathcal{F}_{#1}}
\newcommand{\chquo}{\sslash_{Ch}}
\newcommand{\edgecap}[6]{%
\draw (#1) -- (#2)
node[pos=0.15,sloped,#6,font=\scriptsize] {#3}
node[pos=0.5, sloped,#6,font=\scriptsize] {#4}
node[pos=0.85,sloped,#6,font=\scriptsize] {#5};
}

\title{Chow Quotient Moduli Space of Dynamical Systems on the Projective Line and Rescaling Limits}
\auth
\date{}

\begin{document}
\maketitle
\begin{abstract}
We study Chow-quotient compactifications of the moduli space of degree-$d$ rational self-maps of the projective line. We compute the degree of every three-dimensional orbit cycle in the Clebsch--Gordan compactification in terms of the main component and the hole depths. As consequences, we obtain the generic orbit class and a rational map on Chow quotients induced by iteration. We also introduce a marked Chow quotient designed to record collections of rescaling limits, describe it by decorated trees, and extend the fixed-point multiplier map to the normalization of the unmarked Chow quotient.
\end{abstract}
\section{Introduction}
Let $k$ be an algebraically closed field of $\chara k = 0$ and $d \geq 2$ an integer. The moduli space of the dynamical systems $\fct{rat}_d := \mrat_d/\pglt$ is the quotient of the parameter space $\mrat_d$ of rational self-maps of degree $d$ on $\pone$ with respect to the conjugation action of $G := \operatorname{Aut}(\pone ) = \pglt$.

As in moduli spaces of other objects, boundary points are expected to classify degenerations. In a dynamical setting, rescaling limits (\cite{kiwi2015rescaling}) provide an appropriate notion of degeneration.
For a meromorphic family of dynamical systems $\varphi_t(z)$ regarded as a $\cpx((t))$-point of $\mrat_d$, a rescaling limit is a rational map $\psi$ of degree at least two obtained as
\[ \psi=\lim_{t\to0}(\gamma_t^{-1}\circ\varphi_t^q\circ\gamma_t),\qquad q\geq1, \]
for a parametrized conjugation $\gamma_t\in G(\cpx((t)))$.
For a given family $\varphi_t(z)$, the combination of dynamical systems that appear as rescaling limits is summarized as a dynamical system on \emph{a tree of spheres} (\cite{fujimura-taniguchi2013rational}, \cite{Arfeux2017DynTreeSphere}).
From the Berkovich perspective, one obtains many constraints on combinations of renormalizations (\cite{DeMarco2005IterAtBoundary}, \cite{DeMarcoFaber2014DegenCpxDynSys}, \cite{kiwi2015rescaling}, \cite{Rumely17}).

A basic problem is to classify possible combinations of degenerations in a single moduli space. The algebraic compactifications constructed by geometric invariant theory (\cite{silverman1996p1moduli}, \cite{schmitt2017cptf_stmaps}) do not by themselves retain enough boundary data to describe every combination of rescaling limits (\cite{kiwi-nie2023indet-loci}).

The Deligne-Mumford moduli space $\ol{M}_{0,n}$ of the stable rational curves is a successful example of controlling combinations of different degenerations, which parameterizes trees of spheres: projective lines connected by points with marked points.

Analogous moduli spaces for trees of spheres with dynamical systems are obtained by marking points: fixed points in \cite{fujimura-taniguchi2013rational}, three or more points together with their inverse images in \cite{Arfeux2017CptfTreesOfSpheresCovers}, and dynamical portraits in \cite{DoyleSilverman2020ModuliSpDynSysPortraits}. These markings track the degeneration near each point. However, applying intersection-theoretic tools requires an analysis of their embeddings into moduli spaces of the form $\ol{M}_{0,n}$. More universal boundary constructions using Berkovich geometry have recently been proposed (\cite{FavreGong2025non}, \cite{Roy2025ActionGroupeCptfHybarXiv}); their boundary points retain substantially more non-Archimedean information than the scaling data considered here.

We focus on the Chow quotient construction 
\[ \ol{M}_{0,n} = (\pone )^{n} \chquo G \]
from \cite{kapranov1993ChowQuotGrassI}. After reviewing the general definition and properties of the Chow quotient in \Cref{sec:ChowQuot}, we define the moduli space 
\[ Y_d := \crat_d \chquo G, \]
which parametrizes fixed-point rescalings (the case $q=1$ of Kiwi's rescalings) as well as degenerating rescalings with certain hole configurations. It is equipped with rational iteration maps.

The perturbation--translation--specialization (PTS) principle of \cite{YiHu2005ChowQuotTop} explains how components of a Chow limit can be detected after passing to a valued field. In the present setting, perturbation means lifting to a generic point over a valued extension, while translation and specialization correspond to conjugation and reduction. We use only the component-generic form stated in \Cref{prop:PTS}.

The computation of cycle degrees of $G$-orbits in $\crat_d$ is crucial. We obtain the orbit degree from two fundamental invariants of $[F]\in\crat_d$: the main component $F^{\mathrm{main}}$ and the hole depth $\depth_P[F]$ at each $P\in\pone$, in the terminology of \cite{DeMarco2005IterAtBoundary} (see also \cite{kiwi2015rescaling}).
\begin{theorem}[\Cref{cor:OrdDepth}, \Cref{thm:degorbit}] Let $[F] \in \crat_d$ satisfy $\dim G\cdot[F]=3$. Write $\ol{G\cdot[F]}_{\mathrm{red}}$ for the reduced orbit closure and
\[ \ol{G\cdot[F]}:=(\#\stab_G[F])\ol{G\cdot[F]}_{\mathrm{red}} \]
for the orbit cycle. Their degrees as the algebraic cycles of $\crat_d \simeq \proj^{2d+1}$ satisfy
\[ \deg \ol{G \cdot [F]} = (\#\stab_G[F])\deg \ol{G \cdot [F]}_{\mathrm{red}}
= (d^3-d) -  \sum_{P \in \pone} \delta_{d+1}(c_P),  \]
where
\eqns{
    \delta_{d+1}(c) = c(c-1)(3d+1-2c),\ c_P = \fct{depth}_P[F] + \begin{cases}1 & (\text{if }F^{\text{main}}(P) = P ), \\
    0 & (\text{otherwise}). \end{cases}
}
\end{theorem}
For $[F]\in\mrat_d$, the stabilizer is finite and the correction terms vanish. Thus
\[ \deg\ol{G\cdot[F]}_{\mathrm{red}}=\frac{d^3-d}{\#\fct{Aut}(F)},
\qquad \deg\ol{G\cdot[F]}=d^3-d. \]
In particular, the orbit-cycle morphism extends across the locus of maps with nontrivial automorphisms and induces an injective morphism $\fct{rat}_d\to Y_d$; see \Cref{cor:autfreeChow}.
The stabilizer-weighted degree $d(d+1)(d-1)$ for maps without holes also appears in \cite[Theorem 1.2]{Deopurkar2026EquivClassOrbCloGL2}. Here we compute the correction terms at the boundary and use the orbit cycles to study Chow quotients and iteration.
\begin{corollary}[\Cref{cor:iterdeg}] The iteration map on the parameter space 
\[ \Phi_n : \crat_d \ratmap \crat_{d^n} : \varphi  \mapsto \varphi ^n \qquad (n\geq2) \]
induces a rational map
\[ \Psi_n : Y_d \ratmap Y_{d^n}. \]
\end{corollary}
An ingredient of the PTS principle is orbit-wise iteration on the moduli spaces. We recall from \cite{DeMarco2005IterAtBoundary} that the indeterminacy locus of $\Phi_n$ is independent of $n$ and is
\[ I(d) := \{ F \in \crat_d \mid F^{\mathrm{main}} \text{ is the constant map }F^{\mathrm{main}}(z) = c,\ \fct{depth}_c F \geq 1 \}. \]
Each point $Z$ of the Chow quotient $Y_d$ indicates a finite formal sum of reduced $G$-orbit closures, counted with their Chow multiplicities. The following gives component-wise compatibility of supports with the iteration rational map.
\begin{corollary}[\Cref{cor:iterorbit}]
    Let $(Z,Z')$ be a point of the graph variety $\Gamma_{\Psi_n} \subset Y_d \times Y_{d^n}$ of $\Psi_n$. If a point $[F] \in \crat_d \setminus I(d)$ is included in the family of conjugacy classes which the point $Z$ indicates, then $\Phi_n ([F])$ is included in the family of conjugacy classes which the point $Z'$ indicates.
\end{corollary}
Furthermore, as an analogue of the moduli spaces in \cite{fujimura-taniguchi2013rational}, the normalization of $Y_d$ has the fixed point multiplier morphism.
\begin{theorem}[\Cref{thm:multwelldef}] Let $\nu_d:Y_d^\nu\to Y_d$ be the normalization. The fixed point multiplier map $\Lambda_{1,d} : \fct{rat}_d \to \af^{d+1}$ extends to a morphism $\ol{\Lambda}_d^\nu : Y_d^\nu \to \proj^{d+1}$.
\end{theorem}

For the analysis of $Y_d$, we take a marking extending fixed point marking in \cite{fujimura-taniguchi2013rational}: a representational marking from the $\pglt$-equivariant isomorphism (\cite{west2015moduli}, the assumption $\chara k = 0$ is used in this isomorphism)
\[ \crat_d \simeq \proj (V_{d+1} \oplus V_{d-1} ), \]
where $V_{n} = H^0(\pone , \oshf (n)) = k[x_0 ,x_1]_n$.
The marking comes from the factorization of these representational divisors, 
\eqns{ \pmcrat_d := \left\{ \begin{aligned}
    ((h_1 ,\ldots , h_{2d}), [f_+ : f_-]) \in \proj (V_1)^{2d} \times \crat_d \ \ \  \\ : \prod_{i = 1}^{d+1} h_i \text{ divides } f_+,\ \prod_{i = d+2}^{2d} h_i \text{ divides }f_-
\end{aligned} \right\} }
and take the Chow quotient 
\[  X_d := \pmcrat_d \chquo G .\]
From the $\pglt$-equivariant projections
\[ \Pi : \pmcrat_d \to \crat_d \text{ and } \pi_B : \pmcrat_d \to (\pone)^{2d}, \]
we obtain pushforward morphisms on the Chow quotients 
\[ \Pi_* : X_d  \to Y_d \text{ and }(\pi_B)_* : X_d \to \ol{M}_{0,2d}. \]
By considering the images of $(\pi_B)_*$, we obtain the stratification of $X_d$.
\begin{theorem}[\Cref{cor:stratavariety}, \Cref{thm:Ydatum}] The decomposition stratification of the Chow quotient $X_d := \pmcrat_d \chquo G$ admits an explicit refinement by subvarieties indexed by the strata trees of \Cref{def:stratatree}. Moreover, its points are in bijection with the Chow data of \Cref{def:Chowdatum}.
\end{theorem}

The remainder of the paper is organized as follows.
\Cref{sec:preliminaries} collects preliminaries on trees, valuative criteria, and the Berkovich projective line.
\Cref{sec:ChowQuot} reviews Chow quotients and establishes the component-generic PTS principle and compatibility with equivariant maps.
\Cref{sec:dynsysPone} introduces the Clebsch--Gordan coordinates and representational markings, describes orbit boundaries, and computes marked orbit classes.
\Cref{sec:stratification} describes the marked Chow quotient by decorated trees and proves the regularity of the marked fixed-point multiplier maps.
\Cref{sec:unmarked} studies the unmarked Chow quotient, including the multiplier map on its normalization, isomers, orbit degrees, and rational iteration maps.
\Cref{sec:examples} illustrates the tree description in degree two.

\subsection*{Acknowledgements}
This work was supported by a JSPS Grant-in-Aid for JSPS Fellows (Grant No.~25KJ0090). The author is deeply grateful to Y\^usuke Okuyama for suggesting the moduli-theoretic problem that motivated this work and for helpful comments. The manuscript was prepared with assistance from ChatGPT (GPT-5.6) \cite{OpenAI2026ChatGPT} and Prism \cite{OpenAI2026Prism}, both developed by OpenAI.

\section{Preliminaries}
\label{sec:preliminaries}
\subsection{Trees}

We denote the set $\{ 1, \ldots , n \}$ by $[n]$. For a map $f : S \to T$, we define the fiber partition of $S$ induced by $f$ by $\fib{f} := \{ f^{-1}(t) \subset S \mid t \in f(S) \}$.

\emph{A (simple) graph} is a tuple $\Gamma = (V(\Gamma) , E(\Gamma ))$ of a set $V(\Gamma )$ and a subset $E(\Gamma )$ of $\binom{V(\Gamma )}{2} := \{ S \subset V(\Gamma ) \mid \# S = 2 \}$. An element of the set $V(\Gamma )$ (resp. $E (\Gamma )$) is called \emph{a vertex} (resp. \emph{an edge}) of the graph. For a vertex $v$, \emph{the valence} (also called the degree) of $v$ is the number $\fct{val}(v) := \# \{ e \in E(\Gamma ) \mid v \in e \}$. \emph{A path} on a graph is a finite sequence $(v_0,...,v_n)$ of vertices such that $\{ v_{i-1},v_i\}\in E(\Gamma )$ for each $i\in [n]$. A path $(v_0,\ldots , v_n )$ is called \emph{simple} if $\# \{ v_0 ,\ldots , v_n \} = n+1$. A graph $\Gamma$ is called a \emph{tree graph} if any pair $(v,v')$ of distinct vertices has a unique simple path such that $(v_0=v,v_1, \ldots , v_n = v')$.

A leaf of a tree graph $\Gamma$ is a vertex $v \in V(\Gamma )$ such that $\val (v) = 1$, and we write $L(\Gamma ) := \{ v \in V(\Gamma ) \mid \fct{val}(v) = 1 \}$ for the set of the leaves. A leaf-marked tree is a tuple $(\Gamma ,p)$ of a tree graph $\Gamma$ and a bijection $p : S \to L(\Gamma )$.

We also recall some definitions of $\real$-tree from \cite[1.4]{BakerRumely2010PotTheoryDynOnBerkoP1}. An $\real$-tree is a metric space $(T,d)$ such that any two points $x,y$ are joined by a unique geodesic segment denoted by $[x,y]$, which is homeomorphic to the closed interval. A point $x$ on an $\real$-tree $T$ is called \emph{ordinary} if $T\setminus \{x\}$ has two connected components; otherwise, it is called a \emph{branch point}. For a non-accumulating set $S$ of points and half-lines (represented by its limit points) on an $\real$-tree $T$ including all branch points, the tuple $( S, \{ \{ x, y\} \mid [x,y] \cap S = \{x,y \} \} )$ is a tree graph we denote by $\Gamma (T)$. Conversely, a tree graph $\Gamma$ with an assignment $d : E ( \Gamma )  \to  ( 0 , \infty ]$ of lengths for edges gives an $\real$-tree $T(\Gamma ,d )$, here $d(e) = \infty$ is admitted if $e \cap L(\Gamma ) \neq \emptyset$, which corresponds to half-lines of $\real$-trees. We write $T(\Gamma )$ for the underlying topological space of $T(\Gamma , d)$ for any $d$ of finite values.

For a topological space $X$, we write $\cc (X)$ for the connected components of $X$, equipped with the canonical projection $X \to \cc (X)$.

Let $\Gamma $ be a tree and $p : S \to V(\Gamma )$ a map (for example, the leaf-marked tree $(\Gamma , p : S \to L(\Gamma ))$). For each edge $e = \{ v, v' \} \in E(\Gamma )$, \emph{the edge-contracted marking} $p_e : S \to \{ v, v' \}$ is the composition of the canonical maps
\[ S \to V(\Gamma ) \to T(\Gamma \setminus \{ e\} ) \to \cc ( T(\Gamma \setminus \{ e\} ) ) \simeq \{ v,v' \}. \] Similarly, for each vertex $v \in V(\Gamma )$, \emph{the vertex-contracted marking} $p_v : S \setminus p^{-1}(v) \to E_v := \{ e \in E \mid v \in e \}$ is defined by
\[ S\setminus p^{-1}(v) \to V(\Gamma) \setminus \{ v\}  \to T ( \Gamma \setminus \{ v \} ) \to \cc ( T(\Gamma \setminus \{ v \} )) \simeq E_v . \]
We call the fiber set $\fib{p_e}$ (resp. $\fib{p_v}$) \emph{the edge (resp. vertex) decomposition} of $p$ with respect to $e$ (resp. $v$).

We can reconstruct a leaf-marked tree from its edge decompositions, or its vertex decompositions.
\begin{definition} Let $S$ be a set.
    A \emph{split} of $S$ is an unordered pair $\{A_1,A_2\}$ of nonempty disjoint subsets such that $A_1\sqcup A_2=S$.
\end{definition}

\subsection{Valuative Criterion and Rational Map}

For valuation tests, let $K=k((t))$ or a finite valued extension of it. Thus $K$ is a complete discretely valued field with residue field $k$. We use the extension of the norm given by
\[ |a|=1\quad(a\in k^\times),\qquad |t^n|=\exp(-n),\qquad v_K=-\log|\cdot|. \]
Put $A:=\oring_K$ and let $\maxid_A$ be its maximal ideal.

The valuative criterion of properness says that, for any proper algebraic variety $X$ over $k$, any morphism $x : \spec K \to X$ (equivalently, a $K$-point of $X$) extends uniquely to a morphism $\widetilde{x} : \spec A \to X$. The specialization of $x$ is the composition morphism 
\[ \text{(we denote by) } \spe x : \spec A/\maxid_A \to \spec A \xto{\widetilde{x}} X . \]

For a projective $k$-variety viewed as a compactification of an open subvariety, we can use specializations of $K$-points to represent the points on the boundary as follows.
\begin{proposition} \label{prop:valucriterion} Let $X$ be a projective variety over $k$ and $U$ a dense open subvariety of $X$. For any point $x \in X(k)$, there exists $x_t \in U(K)$ such that $\spe x_t = x$.
\end{proposition}

A rational map $\Phi:X\ratmap Y$ is represented by a morphism $\Phi_U:U\to Y$ on a dense open subvariety $U\subset X$, with representatives identified when they agree on their overlap. We identify $\Phi$ with its unique maximal representative on $U(\Phi)$ and call $I(\Phi):=X\setminus U(\Phi)$ its indeterminacy locus.

The graph closure describes the possible valuative limits of a rational map.
\begin{proposition} \label{prop:indetlocus}
Let $X$ be a projective variety over $k$, let $Y$ be a proper variety over $k$, and let $f:X\ratmap Y$ be a rational map. Fix a dense open subvariety $U\subset U(f)$, and let $\Gamma_f\subset X\times Y$ be the closure of the graph of $f|_U$, with projections
\[ p:\Gamma_f\longrightarrow X,\qquad q:\Gamma_f\longrightarrow Y. \]
For $x\in X(\ol{k})$, the set $q(p^{-1}(x)(\ol{k}))$ is precisely the set of specializations
\[ \left\{\spe f(x_L)\ \middle|
\begin{aligned}
&L/K\text{ is a finite valued-field extension},\\
&x_L\in U(L),\quad \spe x_L=x
\end{aligned}\right\}. \]
In particular, if $f$ is regular at $x$, this set is the singleton $\{f(x)\}$. Conversely, if $X$ is normal and $p^{-1}(x)$ has a single geometric point, then $f$ is regular in a neighborhood of $x$. Thus, when $X$ is normal, $x\in I(f)$ if and only if two such valued lifts of $x$ have distinct specializations under $f$.
\end{proposition}
\begin{proof}
Properness of $Y$ gives one inclusion by specializing the graph point $(x_L,f(x_L))$. Conversely, choose an integral curve through $z\in p^{-1}(x)$ with generic point in the graph over $U$. Normalizing and completing this curve at a point over $z$ gives, after a finite extension of $K$, a valued point $(x_L,f(x_L))$ specializing to $z$.

If $f$ is regular near $x$, then $p^{-1}(x)=\{(x,f(x))\}$. Conversely, a singleton geometric fiber makes $p$ quasi-finite near $x$. Since $p$ is proper and birational, it is finite birational there, hence an isomorphism when $X$ is normal; then $f=q\circ p^{-1}$ is regular near $x$.
\end{proof}
\begin{remark} The normality assumption in the converse cannot in general be omitted: a proper finite birational morphism with singleton geometric fibers need not be an isomorphism over a nonnormal variety.
\end{remark}

\subsection{The Berkovich Projective Line}
Let $\ol K$ be an algebraic closure of the discretely valued field $K$, equipped with the extended valuation. For analytic constructions we use its completion $\widehat{\ol K}$. Throughout, $\ponean_K$ denotes the Berkovich projective line over $\widehat{\ol K}$; its residue field is still $k$. Its points are classified into four types (\cite{BakerRumely2010PotTheoryDynOnBerkoP1}). Apart from the point at infinity, points of types I, II, and III are identified with closed discs in $\widehat{\ol K}$ of radius $0$, positive radius in $\exp(\fct{Im}(v_{\widehat{\ol K}})\cap\real)$, or any other positive radius, respectively. Type~IV points, which we do not use, are represented by infinite decreasing sequences of closed discs in $\widehat{\ol K}$.

The Berkovich upper half plane $\mathcal{H}_K$ is the subspace of $\ponean_K$ of the points of type II or III. The space $\mathcal{H}_K$ has a natural $\real$-tree structure which can be described by the gluing of isometries 
\[ \alpha_x : \real \to \ponean_K: r\mapsto D(x,\exp (r) ) \text{ for each }x \in \widehat{\ol K}. \]

A nonconstant morphism $\varphi : \pone_K \to \pone_K$ induces the morphism $\varphi^{an}: \ponean_K \to \ponean_K$ and types of points are preserved under the induced morphism.

\section{Chow Quotient} \label{sec:ChowQuot}
\subsection{Chow scheme and Definition of Chow Quotient}
We recall some fundamental facts about Chow schemes (\cite{gkz1994disc-res-mult}, \cite{Kollar1996RatCurvesOnAlgVar}) and Chow quotients (\cite{kapranov1993ChowQuotGrassI}).

Let $X$ be a projective variety.
\emph{An algebraic cycle} $Z$ of dimension $m$ on $X$ is a formal finite sum $Z = \sum_i n_i Z_i$ of $m$-dimensional irreducible closed subvarieties $Z_i$ of $X$ with integral coefficients. \emph{The support} of an algebraic cycle $Z$ is a closed subset of $X$ defined as $\fct{supp} (Z) := \bigcup_i \fct{supp}(Z_i)$. We write $Z_m(X)$ for the abelian group of the algebraic cycle of dimension $m$ on $X$.

The Chow scheme $\fct{Chow}(\proj^n,m,d)$ parametrizes effective cycles of dimension $m$ and degree $d$ in $\proj^n=\proj(V)$, with projective embedding
\eqns{
\fct{Chow}(\proj^n ,m,d) & \hookrightarrow \proj ( H^0( Gr(n-m, V) , d \cdot H ) ) \\
Z & \mapsto R_Z : V_+(R_Z) = \{ [W] \in Gr(n-m,V) \mid W \cap Z_i \neq \emptyset \} \\ 
 & \ \ \ \ ( \text{if }Z : \fct{irreducible}), \\
\sum n_i Z_i & \mapsto \prod R_{Z_i}^{n_i} \ (\text{otherwise}), 
}
where $H$ is the hyperplane class, the pullback of $\oring (1)$ under the Pl\"{u}cker embedding $Gr(n-m,V) \to \proj (\bigwedge^{n-m} V)$. Thus the set of the effective algebraic cycles on $\proj^n$ has a scheme structure 
\[ \chow (\proj^n ) = \bigsqcup_{m,d} \chow (\proj^n , m,d ). \]

Let $X$ be a projective variety. For any closed embedding $X \hookrightarrow \proj^n$, the Chow scheme of $X$ is regarded as a closed subscheme of $\chow (\proj^n )$
\[ \fct{Chow}(X) := \{ Z \in \fct{Chow}(\proj^n ) \mid \supp Z \subset X \}, \]
and the induced scheme structure is independent of the embedding.

Following the dimension convention for Chow groups, put
\[ A_m(X):=Z_m(X)/\mathord{\sim}_{\mathrm{rat}}. \]
Write $[Z]\in A_m(X)$ for the rational-equivalence class of $Z$ and, when $X$ is pure of dimension $N$, put $A^c(X):=A_{N-c}(X)$. For the general Chow-quotient construction we instead use algebraic equivalence, writing
\[ A_m^{\mathrm{alg}}(X):=Z_m(X)/\mathord{\sim}_{\mathrm{alg}},\qquad [Z]_{\mathrm{alg}}\in A_m^{\mathrm{alg}}(X). \]
For $\delta\in A_m^{\mathrm{alg}}(X)$, write
\[ \chow(X,\delta):=\{Z\in\chow(X)\mid Z\text{ is an effective }m\text{-cycle},\ [Z]_{\mathrm{alg}}=\delta\}. \]
For projective spaces and projective bundles over products of projective lines, the projective bundle formula gives a perfect integral intersection pairing on the Chow ring. Since algebraically trivial cycles are numerically trivial, rational and algebraic equivalence coincide on these varieties. We identify the two groups in our computations below, retaining $A^*(X)$ for the rational-equivalence Chow ring.

For a morphism $f:X\to Y$ between projective schemes and an effective cycle $Z=\sum_i n_i[Z_i]$, proper pushforward is given by
\eqnl{ f_*Z=\sum_{\dim f(Z_i)=\dim Z_i}n_i\deg(f|_{Z_i})[f(Z_i)]. \label{eq:chowpo} }
On a component of the Chow scheme on which the pushforward class is constant, this cycle-theoretic operation induces the corresponding morphism of Chow schemes.

In this section, let $G$ be a connected and reduced algebraic group acting on $X$.
We assume that the action of $G$ is generically free, that is, the stabilizer group is trivial for a generic point $x$ of $X$. We write $\ol{G\cdot x}_{\mathrm{red}}$ for the reduced closed subvariety underlying the orbit closure. If $x$ has finite stabilizer, we define the \emph{orbit cycle} by
\eqnl{ \ol{G\cdot x}:=(\#\stab_G(x))\ol{G\cdot x}_{\mathrm{red}}\in Z_{\dim G}(X). \label{eq:orbitcycle} }
Thus the notation without the subscript $\mathrm{red}$ always includes the stabilizer multiplicity. We denote by $\delta(G,x):=[\ol{G\cdot x}]_{\mathrm{alg}}\in A_{\dim G}^{\mathrm{alg}}(X)$ the algebraic-equivalence class of the orbit cycle. For an irreducible subvariety $X'\subset X$ whose generic point has finite stabilizer, this class is constant on a dense open subset; we denote it by $\delta(G,X')$. Choose a $G$-invariant dense open subset $U=U(G,X)\subset X$ on which the action is free and the orbit closures form an algebraic family. This gives the orbit-cycle morphism
\[ \pi_U : U \to \chow (X, \delta (G,X)) : x \mapsto \ol{G \cdot x}. \]
\begin{definition}[\cite{kapranov1993VeroneseAndM0n}] 
\emph{The Chow quotient} of $X$ by $G$ is defined by 
\[ X \sslash_{Ch} G := \ol{\pi_U (U)} = \ol{ \{ \ol{G\cdot x}^{X} \mid x\in U \} }^{\fct{Chow}(X)}. \]
\end{definition}
The Chow quotient has the universal cycle $Z \to X$, the $X \sslash_{Ch} G$-family of cycles of class $\delta (G,X)$ of $X$.

\subsection{Chow Families: Lifting and Local Constructions}

Let $X$ be a projective variety and $\delta$ a cycle class of $X$. The Chow scheme $\chow (X,\delta )$ has the universal fiber, which is called the \emph{Chow fiber}
\[ \cat{Z} = \cat{Z}_{X,\delta} := \{ (Z , x) \mid x \in \supp (Z) \} \subset \chow (X,\delta ) \times X. \]
We remark that $\cat{Z}$ is just a Zariski closed subset of $\chow (X,\delta ) \times X$.

The following lifting lemma does not prescribe the specialization in the fiber.
\begin{lemma} \label{lem:curvefib}
    Let $R$ be a discrete valuation ring and $f : X \to Y$ a surjective morphism between projective varieties. For any morphism $y_R : \spec R \to Y$, there exists a finite extension $S$ of $R$ and $x_S : \spec S \to X$ such that $f(x_S) = y_R$.
\end{lemma}
\begin{proof}
After a finite extension of the fraction field of $R$, choose a point of the generic fiber of $X\times_Y\spec R\to\spec R$. The valuative criterion for properness extends this point to the valuation ring $S$ of an extension of the valuation, giving the required lift.
\end{proof}
For Chow families, we can also prescribe a dense open subset of a special-fiber component.
\begin{proposition} \label{prop:PTS}
    Let $R$ be a discrete valuation ring with closed point $s$, and let $\varphi:\spec R\to\fct{Chow}(\proj^n,m,d)$ be a morphism. Fix an irreducible component $C\subset\supp\varphi(s)$ and a dense open subset $U\subset C$. After a finite extension $S/R$, there is a morphism
    \[ \widetilde{\varphi}:\spec S\longrightarrow\proj^n \]
    lying in the support of the corresponding cycle on both fibers, with specialization in $U$.
\end{proposition}
\begin{proof}
Let $\eta$ be the generic point of $\spec R$, write $\varphi(\eta)=\sum_i n_i[V_i]$, and let $\mathcal V_i\subset\proj^n_R$ be the scheme-theoretic closure of $V_i$. Each $\mathcal V_i$ is integral and flat over $R$. Compatibility of Chow forms with specialization gives
\[ \varphi(s)=\sum_i n_i[(\mathcal V_i)_s]_m, \]
where $[\,\cdot\,]_m$ denotes the dimension-$m$ fundamental cycle. Thus $C$ is a component of $(\mathcal V_i)_s$ for some $i$. Cutting $\mathcal V_i$ by general hyperplanes through a general closed point of $U$ gives a horizontal integral curve through that point. Normalizing this curve and taking the valuation ring at a point above it gives the required finite extension $S/R$ and lift.
\end{proof}

We record a local criterion for recovering the summands of a Chow cycle.
\begin{lemma}[Recovery of cycle summands] \label{lem:ChowFactorization}
Let $X$ be a projective variety and $\mathcal C_1,\ldots,\mathcal C_r\subset\fct{Chow}(X)$ projective subvarieties parametrizing effective $m$-cycles of fixed positive degrees. Let $H$ be a permutation group acting only among equal factors. Consider the addition map
\[ a:(\mathcal C_1\times\cdots\times\mathcal C_r)/H\longrightarrow\fct{Chow}(X),
\qquad [(Z_1,\ldots,Z_r)]\longmapsto\sum_i Z_i. \]
Suppose a geometric cycle $Z=\sum_i Z_i$ has a unique such decomposition up to $H$, and no two summands share an irreducible component. Then $a^{-1}(U)\to U$ is a closed immersion for some open neighborhood $U$ of $Z$ in $\fct{Chow}(X)$.
\end{lemma}
\begin{proof}
The map $a$ is finite: it is proper, and a fixed effective cycle has only finitely many decompositions into summands of the prescribed degrees.

In Chow coordinates, addition is multiplication of forms on a Grassmannian. For two forms $F,G$ with no common divisor component, a tangent vector in the kernel satisfies, after choosing scalar representatives,
\[ \dot F G+F\dot G=0,\qquad \frac{\dot F}{F}=-\frac{\dot G}{G}. \]
The two fractions can have poles only on disjoint sets of divisor components, so their common value has no divisorial poles. Normality and projectivity of the Grassmannian make it constant, hence the tangent vector is zero in projective coordinates. Induction gives the same conclusion for all factors. Since $H$ acts freely here, $a$ is unramified at the unique point over $Z$.

As $a$ is finite, we may shrink around $Z$ so that $a$ is unramified throughout. The singleton geometric fiber then gives the closed immersion by \cite[\href{https://stacks.math.columbia.edu/tag/04DG}{Tag 04DG}]{stacks-project}.
\end{proof}

Degree-one intersections give another local construction.
\begin{lemma}[Degree-one Chow slice] \label{lem:degreeoneChowSlice}
Let $X$ be a projective variety, let $\mathcal C$ be a projective subvariety of a component of $\fct{Chow}(X)$ parametrizing $m$-dimensional cycles, and let $D_1,\ldots,D_m$ be effective Cartier divisors on $X$. On the open subset $\mathcal C^{\circ}\subset\mathcal C$ where the successive intersections with the $D_r$ are proper, the assignment
\[ Z\longmapsto Z\cdot D_1\cdots D_m \]
defines a morphism from $\mathcal C^{\circ}$ to the Chow scheme of zero-cycles on $X$. If these zero-cycles have degree one, this is a morphism
\[ \tau_{D_1,\ldots,D_m}:\mathcal C^{\circ}\longrightarrow \fct{Chow}(X,0,1)\simeq X. \]
\end{lemma}
\begin{proof}
The hyperplane substitution and hypersurface formulas of \cite[\S2.1, Propositions 2.1--2.2]{DalbecSturmfels1995ChowForms}, expressed in symmetric coordinates and normalized by a nonzero coefficient, give regular Chow coordinates on the proper-intersection locus.

For an effective Cartier divisor $D$ and a cycle $Z_0$ meeting it properly, choose an effective $E$ meeting $Z_0$ properly with $E$ and $D+E$ very ample. The corresponding embeddings give $Z\cdot E$ and $Z\cdot(D+E)$ regularly near $Z_0$. Fix any Chow embedding. Then
\[ R_{Z\cdot(D+E)}=R_{Z\cdot D}\,R_{Z\cdot E}. \]
Multiplication by the known nonzero factor $R_{Z\cdot E}$ is injective on forms of the required degree; an invertible maximal minor therefore recovers the quotient coefficients regularly on this divisibility locus. Iterating proves the assertion, and $\fct{Chow}(X,0,1)\simeq X$ gives the degree-one case.
\end{proof}

\subsection{PTS principle, for Components of Chow Quotients}

For algebraic cycles $Z,Z'$, write $Z'\leq Z$ when $Z-Z'$ is effective.

\begin{proposition} \label{prop:stabilizermultiplicity}
    Let $X$ be a projective variety and $G$ a connected reductive group with generically free action on $X$. For any $Z \in X \chquo G$ and $x \in \supp Z$, if $x$ has a finite stabilizer group $H$, then we have 
    \[ \ol{G \cdot x} \leq Z . \]
\end{proposition}
\begin{proof}
Let $\ol G$ be a projective equivariant compactification of $G$ for the left-multiplication action, and let $G$ act diagonally on $\ol G\times X$. Choose a valuation curve of generic orbit cycles specializing to $Z$. Since $x$ has finite stabilizer, its orbit is dense in the $\dim G$-dimensional component $\ol{G\cdot x}_{\mathrm{red}}$ of $\supp Z$. By \Cref{prop:PTS}, after a finite extension and a constant translation we may represent the generic cycles by $x_K\in X(K)$ with $\spe x_K=x$.

Let $\widetilde Z$ be the specialization of the orbit closure of $(e,x_K)$ in $\ol G\times X$. The orbit of $(e,x)$ has trivial stabilizer, so
\[ C:=\ol{G\cdot(e,x)}_{\mathrm{red}} \]
occurs in $\widetilde Z$ with coefficient at least one. Moreover,
\[ \fct{pr}_{2,*}\widetilde Z=Z. \]
The restriction $\fct{pr}_2|_C:C\to\ol{G\cdot x}_{\mathrm{red}}$ is generically finite of degree $\#\stab_G(x)$. Applying the component-wise pushforward formula \cref{eq:chowpo} therefore shows that $Z$ contains
\[ (\#\stab_G(x))\ol{G\cdot x}_{\mathrm{red}}=\ol{G\cdot x}, \]
as asserted.
\end{proof}

Let $X_1$ and $X_2$ be projective varieties with $G$-actions and let $Y_i := X_i \sslash_{Ch}G$ ($i = 1,2$).
For a $G$-equivariant morphism $\Phi : X_1 \to X_2$, 
if the equality of algebraic-equivalence cycle classes
\eqnl{\delta (G , \fct{Im}(\Phi )) = \delta (G , X_2 ) \label{eq:dGim} } 
is satisfied, the pushforward of Chow cycles induces the morphism 
    \[ \Phi_* : Y_1 = X_1 \chquo  G  \to Y_2 = X_2 \chquo G. \]
This equality is automatically satisfied if $\Phi$ is surjective. Moreover, if $X_1 \to X_2$ is a birational morphism, the equality \cref{eq:dGim} is automatically satisfied and moreover $\Psi : Y_1 \to Y_2$ is also a birational morphism. We state this fact as a proposition for reference.
\begin{proposition} \label{prop:morphChquo}
 Let $\Phi : X_1 \ratmap X_2$ be a $G$-equivariant rational map between projective $G$-varieties. If $\delta (G, \pi_2( \Gamma_\Phi ) ) = \delta (G, X_2)$ for the graph variety $\Gamma_\Phi$ of $\Phi$, then the pushforward of cycles induces a rational map $\Phi_* : X_1\chquo G \ratmap X_2 \chquo G$.
\end{proposition}

\begin{remark}
    We apply this construction for moduli space of dynamical system by \Cref{cor:iterdeg}.
\end{remark}

\begin{proposition} \label{prop:ChowPushRat}
    Let $\Phi:X_1\ratmap X_2$ be a $G$-equivariant rational map between irreducible projective varieties, and suppose that it induces a rational map $\Psi:Y_1\ratmap Y_2$ on their Chow quotients. If $(Z_1,Z_2)\in\Gamma_\Psi(k)$ and
    \[ x\in\supp Z_1\setminus I(\Phi) \]
    has finite stabilizer, then
    \[ \ol{G\cdot x}\leq Z_1,
       \qquad \Phi(x)\in\supp Z_2. \]
    If $\Phi(x)$ also has finite stabilizer, then moreover
    \[ \ol{G\cdot\Phi(x)}\leq Z_2. \]
\end{proposition}
\begin{proof}
Choose, after a finite valued-field extension, a valuation curve in $\Gamma_\Psi$ specializing to $(Z_1,Z_2)$ whose generic point is represented by a point $u_K$ in the generic orbit-cycle locus of $X_1$. Since $x$ has finite stabilizer, $\ol{G\cdot x}_{\mathrm{red}}$ is a component of $\supp Z_1$ and $G\cdot x$ is dense in that component. Applying \Cref{prop:PTS} to the first universal cycle and translating by a constant element of $G(k)$, we may choose $\gamma_K\in G(K)$ such that
\[ \spe(\gamma_Ku_K)=x. \]
Because $x\notin I(\Phi)$, equivariance and the valuative criterion for the graph of $\Phi$ give
\[ \spe\Phi(\gamma_Ku_K)=\Phi(x). \]
Thus $\Phi(x)\in\supp Z_2$, and both cycle inequalities follow from \Cref{prop:stabilizermultiplicity} under the respective finite-stabilizer assumptions.
\end{proof}

\begin{definition}
For a multiset $\mathcal{D} := \{ \delta_1 ,\ldots , \delta_n \}$ of classes of effective cycles $\delta_i \in A_{\dim G}^{\mathrm{alg}}(X)$ such that $\sum_i \delta_i = \delta (G,X)$, we define the \emph{$\mathcal{D}$-stratum} of $X \chquo G$ by
\[ ( X \chquo G )_{\mathcal{D}} = \left\{ Z \in X\chquo G \,\middle|\,
\begin{aligned}
& Z = \sum_{i} n_iZ_i,\ Z_i \text{ is irreducible, }\\
& Z_i \neq Z_j (i \neq j),\ [n_iZ_i]_{\mathrm{alg}}=\delta_i
\end{aligned} \right\}. \]
\emph{The decomposition stratification} of $X\chquo G$ is 
\[ X \chquo G = \bigsqcup_{\mathcal{D}} ( X \chquo G )_{\mathcal{D}}. \]
\end{definition}

\section{Dynamical Systems on $\pone$}
\label{sec:dynsysPone}
\subsection{Convention: directions of group actions}

On both the dynamical and multiplier projective lines, we use homogeneous coordinates $[z_0:z_1]$ and the affine coordinate $z=z_1/z_0$. Thus a finite multiplier $\lambda$ is represented by $[1:\lambda]$, and $\infty=[0:1]$. Our M\"obius convention is
\[ \gamma=\begin{bmatrix}a&b\\c&d\end{bmatrix}:
z\longmapsto\frac{az+b}{cz+d},\qquad
[z_0:z_1]\longmapsto[dz_0+cz_1:bz_0+az_1]. \]
Thus matrices act on the column $(z_1,z_0)^t$. The left action on maps is $\gamma\cdot\varphi=\gamma\circ\varphi\circ\gamma^{-1}$, and on binary forms we use inverse substitution with an $\slt$-lift of $\gamma$. In particular, the roots of $\gamma\cdot f$ are the images under $\gamma$ of the roots of $f$.

A standard compactification (\cite{silverman1996p1moduli}) of the moduli space $\mrat_d$ is 
\eqns{ \ol{\mrat}_d \simeq \proj^{2d+1} & = \proj ( H^0( \pone \times \pone , \oshf_{\pone \times \pone} (d,1) )), }
so that a point $[a,b] = [a_0 : \cdots : a_d : b_0 : \cdots : b_d] \in \ol{\mrat}_d$ corresponds to the degree $(d,1)$ divisor
\eqns{ F_{[a,b]} (x_0,x_1,y_0,y_1) & = y_1(a_0x_0^{d} + \cdots + a_dx_1^{d}) - y_0(b_0x_0^{d} + \cdots  + b_dx_1^{d})}
of $\pone_x \times \pone_y$. For the factorization of $F = F_{[a,b]}$ as 
\eqns{
 F = F^{\text{main}}(x_0 , x_1,y_0, y_1) \cdot \prod_i (\alpha_{i,1} x_0 - \alpha_{i,0}x_1), 
}
with irreducible divisor $F^{\text{main}}$ of degree $(d',1)$, the map corresponding to $F^{\text{main}}$ is called the main component of $F_{[a,b]}$. Points $[ \alpha_{i,0} : \alpha_{i,1}] \in \pone$ are called holes of $F$ and its multiplicity in the factorization is called the depth of $F$ at $P$ and denoted by $\fct{depth}_P(F)$. 

The automorphism group $\fct{Aut}(\pone)=\pglt$ acts diagonally on $\pone\times\pone$. The corresponding $\slt$-linear representation on
\[ H^0(\pone\times\pone,\oshf(d,1)) \]
is $V_d\otimes V_1$, and its projectivization gives the $\pglt$-conjugation action on $\crat_d$. The center of $\slt$ acts by the same scalar on $V_{d+1}$ and $V_{d-1}$, and hence trivially on $\proj(V_{d+1}\oplus V_{d-1})$; consequently the $\slt$- and $\pglt$-orbits considered below coincide.

\subsection{Characterization of Covariants under Clebsch-Gordan isomorphisms}
We recall that the group $\slt$ acts on the spaces $V_m := k[x_0, x_1]_m$ of binary forms. In this subsection, we use $x,y,z$ for pairs of variables $(x_0,x_1),(y_0,y_1),(z_0,z_1)$. 

For any pair of pairs of variables $(x,y) = ((x_0,x_1),(y_0,y_1))$, the determinant $\Delta_{xy} :=[x,y] := x_0y_1-y_0x_1$ and the derivative operators 
\[ \Omega_{xy} := \partial_{x_0} \partial_{y_1} - \partial_{y_0} \partial_{x_1},\ \text{ and }\sigma_x^y := y_0 \partial_{x_0} + y_1 \partial_{x_1} \]
are $\slt$-equivariant.

\begin{lemma} \label{lem:CGiso}
    The Clebsch-Gordan isomorphism between $\slt$-representations $\fct{CG}: V_d \otimes V_1 \simeq V_{d+1} \oplus V_{d-1}$ is given by
    \eqns{\begin{array}{cccc}
    \fct{CG}: & ( k [x_0,x_1]_{d} \otimes k [y_0, y_1]_1 = ) V_d \otimes V_1 & \simeq & V_{d+1} \oplus V_{d-1} \\
    & F(x,y) & \mapsto & ( F(z,z) , (\Omega_{xy}F)(z,z)) \\
    & \frac{1}{d+1}( \sigma_{x}^y f_+(x) + \Delta_{xy} f_-(x) )  & \mapsfrom & (f_+(z) , f_-(z)).  \\
    \end{array}
    }
\end{lemma}
\begin{proof}
   The isomorphism can be checked by an explicit computation.
\end{proof}
\begin{remark}
    The isomorphism depends on the assumption $\chara k = 0$.
\end{remark}
In particular, for $f = [f_0 : f_1] \in \ol{\mrat}_d$, $F(x,y) = y_1f_0(x) - y_0f_1(x)$ implies that
\eqnsl{ f_0(x) & = a_0x_0^d + \cdots + a_dx_1^d = \frac{1}{d+1}( \partial_{x_1}f_+(x) + x_0f_-(x) ),  \label{eq:f0f1tofpfn}\\
f_1(x) & = b_0x_0^d + \cdots + b_dx_1^d = \frac{1}{d+1}( -\partial_{x_0}f_+(x) + x_1f_-(x)). \nonumber }

\begin{corollary} \label{cor:OrdDepth}
    Let $P\in \pone$, $[F(x,y)] \in \ol{\mrat}_d$ and $\fct{CG}(F) = (f_+,f_-)$. 
    We have
    \[ \min ( \ord_P f_+ , \ord_P f_-+1) = \fct{depth}_P [F] + \begin{cases}
        1 & (\text{if }F^{\mathrm{main}}(P) = P )\\
        0 & (\text{otherwise}).
    \end{cases}\]
\end{corollary}
    
For a map $\varphi \in \mrat_d$, the multiplier of a fixed point $\alpha \in \pone$ is the scalar by which the differential $d\varphi_\alpha$ acts on $T_\alpha\pone$; in a local coordinate it is the derivative
\[ \lambda_\varphi (\alpha ) := \varphi'(\alpha ). \]
In general, the multiplier of a periodic orbit is defined as the fixed point multiplier of the corresponding iteration. Multiplier values are independent of the choice of coordinate on $\pone$; in particular, their collections are invariant under the $\pglt$-action.
\begin{proposition}
    Let $\varphi = \varphi_F = [F(x,y)] \in \mrat_d$ and $\fct{CG}(F) = (f_+,f_-)$.
    \begin{enumerate}
        \item \label{item:dF} The derivative of $\varphi$ as the rational function $\varphi : \pone \to \pone$ with respect to the coordinate $\ol{z} = z_1/z_0$ is \eqnl{ \label{eq:dF} \varphi'( \ol{z}) = \frac{z_0}{f_0^2} ( f_0 \partial_{1}f_1 - f_1 \partial_{1}f_0) (z_0,z_1). }
        \item \label{item:lFz} The multiplier of $\varphi$ at a fixed point $P \in \pone$ of $F$ is given on the two standard affine charts by
        \eqnl{ \label{eq:lFz}
        \lambda_{\varphi}(P) =
        \begin{cases}
        \displaystyle \frac{ f_- - (d\partial_1f_+)/z_0 }{f_- + (\partial_1 f_+)/z_0}(P) & (z_0(P)\neq 0),\\[6pt]
        \displaystyle \frac{ f_- + (d\partial_0f_+)/z_1 }{f_- - (\partial_0 f_+)/z_1}(P) & (z_1(P)\neq 0).
        \end{cases} }
    \end{enumerate}
\end{proposition}
\begin{proof}
    \cref{item:dF} This is the homogenization of the derivative of a quotient.

    \cref{item:lFz} We recall that $z_1 f_0-z_0 f_1=f_+$ and $f_+(P)=0$ for a fixed point $P=[\alpha_0:\alpha_1]$. On the chart $z_0\neq0$, \cref{eq:dF} gives
    \eqns{
    \lambda_\varphi(P)
    &=\left(\frac{z_0\partial_1f_1-z_1\partial_1f_0}{f_0}
      +\frac{f_+\partial_1f_0}{f_0^2}\right)(P)\\
    &=\frac{f_0-\partial_1f_+}{f_0}(P)
    =\frac{z_0f_--d\partial_1f_+}{z_0f_-+\partial_1f_+}(P),
    }
    where the last equality follows from \cref{eq:f0f1tofpfn}. Dividing the final quotient by $z_0$ gives the first formula in \cref{eq:lFz}.

    On the chart $z_1\neq0$, using the coordinate $z_0/z_1$ instead gives
    \eqns{
    \lambda_\varphi(P)
    &=\left(\frac{z_1\partial_0f_0-z_0\partial_0f_1}{f_1}
      -\frac{f_+\partial_0f_1}{f_1^2}\right)(P)\\
    &=\frac{f_1+\partial_0f_+}{f_1}(P)
    =\frac{z_1f_-+d\partial_0f_+}{z_1f_--\partial_0f_+}(P).
    }
    Dividing by $z_1$ gives the second formula. On the overlap, Euler's identity and $f_+(P)=0$ imply
    \[ \frac{\partial_0f_+}{z_1}(P)=-\frac{\partial_1f_+}{z_0}(P), \]
    so the two expressions agree.
\end{proof}

\subsection{Markings of factors and boundary of orbits}

Let $B := (\pone )^{2d}$, let $H_i$ be the hyperplane divisor for the $i$-th coordinate of $B$, and put $D_+ := H_1 + \cdots + H_{d+1}$ and $D_- := H_{d+2} + \cdots + H_{2d}$. We use the quotient convention for projective bundles and define
\[ X := \pmcrat_d := \proj_B\!\left(\oshf_B(D_+) \oplus \oshf_B(D_-)\right), \qquad \pi_B:X\to B. \]
The products of the universal linear factors give bundle morphisms
\eqns{
\oshf_B(-D_+) &\longrightarrow V_{d+1}\otimes\oshf_B, &
\oshf_B(-D_-) &\longrightarrow V_{d-1}\otimes\oshf_B.
}
Their direct sum is fiberwise injective. Equivalently, after identifying a one-dimensional quotient of
$\oshf_B(D_+)\oplus\oshf_B(D_-)$ with a line in its dual, it induces the product morphism
\eqns{ \Pi : X & \to \proj (V_{d+1} \oplus V_{d-1} ) \simeq  \ol{\mrat}_d \\
( [u:v] , ([\alpha_{i,0}: \alpha_{i,1}])_{i = 1}^{2d} ) & \mapsto \left[  u \prod_{i = 1}^{d+1} (\alpha_{i,1} x_0- \alpha_{i,0}x_1) : v \prod_{i = d+2}^{2d} (\alpha_{i,1} x_0 - \alpha_{i,0}x_1) \right]. }

The morphism $(\pi_B, \Pi) : X \to B \times \ol{\mrat}_d$ is an embedding and we can regard $X$ as
\eqnsl{ X & \simeq \left\{ (b,[f_+:f_-]) \in B \times \ol{\mrat}_d \middle| \Pi_+(b) \text{ divides }f_+,\ \Pi_-(b) \text{ divides }f_-  \right\}, \nonumber \\
& \text{ where }\Pi_+ ( [\alpha_{1,0} : \alpha_{1,1}], \ldots , [\alpha_{2d , 0} : \alpha_{2d,1}]) = \prod_{i=1}^{d+1} (\alpha_{i,1}x_0 - \alpha_{i,0} x_1)  \label{eq:fpfmdiv} \\
& \text{ and }\Pi_- ( [\alpha_{1,0} : \alpha_{1,1}], \ldots , [\alpha_{2d , 0} : \alpha_{2d,1}]) = \prod_{i=d+2}^{2d} (\alpha_{i,1}x_0 - \alpha_{i,0} x_1). \nonumber
}
In order to describe orbits and to compute the cycle classes of orbits, we introduce some combinatorial notation. 

For a point $x = (P_1, \ldots , P_{2d}, [f_+ : f_-] ) \in X(k)$ and $Q \in \pone$, we set
\eqnl{ \label{eq:abcdef}
\begin{aligned}
S_Q(x) & := \{ j \mid Q = P_j,\ j \in [2d] \}\\
& \text{ with a decomposition } S_Q(x) = S^+_Q(x) \sqcup S^-_Q(x) \text{ where }\\
S^+_Q(x) & := S_Q(x) \cap [d+1],\ S^-_Q(x) := S_Q(x) \cap [d+2, 2d], \\
a_Q(x) & := \# S^+_Q(x), \ b_Q(x) := \# S^-_Q(x)\text{ and } \\
c_Q(x) & := \min (\ord_Q f_+ , \ord_Q f_-+1) =\begin{cases} b_Q(x)+1 & (f_+ = 0) \\
a_Q(x) & (f_-=0) \\
\min(a_Q(x),b_Q(x)+1) & (f_+f_- \neq 0).
\end{cases}
\end{aligned}
}
We sometimes omit $x$ if it is apparent, and for any $b \in B = (\pone )^{2d}$, we similarly define $S_Q(b), S_Q^+(b), S_Q^-(b), a_Q(b)$ and $b_Q(b)$.

The $\pglt$-action on $B = (\pone )^{2d}$ is the diagonal action, and thus
\eqn{ \dim G \cdot (P_1,\ldots , P_{2d}) = \min (3, \# \{ P_1,\ldots , P_{2d} \} ). }
In particular, the small diagonal of $B$ is the unique $1$-dimensional orbit we denote by $B_1$, and $B$ has $2^{2d-1}-1$ disjoint $2$-dimensional orbits indexed by the splits $A = \{ A_1, A_2 \}$ of $[2d]$, 
\eqnl{ \label{eq:bsplit} B_A := \{ (P_1 ,\ldots , P_{2d} ) \mid \exists Q_1,Q_2 \in \pone, Q_1 \neq Q_2,\ P_i = Q_j \text{ if }i \in A_j \}.
}
For a split $A = \{ A_1,A_2 \}$ of $[2d]$, we put 
\[ d(A) := | \# (A_1 \cap [1,d+1]) - \# (A_1 \cap [d+2,2d]) - 1|. \]
We write $B_3$ for the open subset of $B$ made of all $3$-dimensional orbits.
Then we can describe the orbit decomposition of $X$ as follows.
\begin{proposition} \label{prop:orbdecomp}
For a point $b \in B$, the orbit decomposition of $\pi_B^{-1}(G\cdot b) ( \subset X)$ is
\eqns{
\pi_B^{-1}(G\cdot b) = \begin{cases}
    \{ f_+ = 0 \} \sqcup \{ f_- = 0 \} \sqcup \{ f_+ f_- \neq 0 \} & \text{if} \begin{cases} b \in B_1 \text{ or}\\
     (b \in B_A \text{ and }d(A) \neq 0),\end{cases}\\
    \bigsqcup_{x \in \pi_B^{-1}(b)} G \cdot x & \text{if} \begin{cases} b \in B_3 \text{ or}\\
     (b \in B_A \text{ and }d(A) = 0).\end{cases}\\
\end{cases}
}
\end{proposition}
\begin{proof}
    We recall that the orbit decomposition of $\pi_B^{-1}(G\cdot b)$ is the orbit decomposition of $\pi_B^{-1}(b)$ with respect to the stabilizer group action.
    We assume that $b \in B_A$ for some split $A = \{A_1 , A_2 \}$. Since the orbit decomposition of fiber does not depend on the choice of a point on orbit, we assume that $b = (P_1,\ldots , P_{2d} ) \in \{ 0 , \infty \}^{2d} ( \subset ( \pone )^{2d})$. By \cref{eq:fpfmdiv}, a point $x \in \pi_B^{-1}(b)$ can be described as
    \[ x = \left( b, [ u x_0^{a_\infty} x_{1}^{a_0} : v x_0^{b_\infty} x_{1}^{b_0}] \right). \]
    The stabilizer group of $b$ is the diagonal torus $G_b = \begin{pmatrix} t &  0\\ 0 & t^{-1}  \end{pmatrix} \subset \slt$, so the action of the stabilizer group is given as 
    \[ \begin{pmatrix} t &  0\\ 0 & t^{-1}  \end{pmatrix} \cdot x = ( b,  [ t^{a_\infty -a_0 }u x_0^{a_\infty} x_{1}^{a_0} : t^{b_\infty - b_0}v x_0^{b_\infty} x_{1}^{b_0}]). \]
    Here we have 
    \[ a_\infty + a_0 = d+1 \text{ and } b_\infty + b_0 = d-1, \]
    thus
    \[ a_\infty - a_0 = b_\infty - b_0 \text{ if and only if } |a_0 - b_0 - 1| = 0. \]
    By \cref{eq:abcdef,eq:bsplit}, $|a_0 - b_0 - 1| = d(A)$ for $x \in \pi_B^{-1}(B_A)$. So $x$ is $G_b$-invariant if and only if $d(A) = 0$. Otherwise, $\pi_B^{-1}(b)$ is decomposed into three orbits $\{ f_+ = 0 \} \sqcup \{f_- = 0 \} \sqcup \{f_+ f_- \neq 0\}$. Here these conditions are stable under the $G$-action, so this also gives the orbit decomposition of $\pi_B^{-1}(G\cdot b)$. The case of $b \in B_1$ is similar.
\end{proof}

\begin{corollary} \label{cor:boundaryorbit}Let $x = (b, [f_+:f_- ] ) \in X$ be a point in $X$ such that $\dim G \cdot b = 3$. Let $x' \in \ol{G \cdot x}_{\mathrm{red}} \setminus G\cdot x$.
    \begin{enumerate}
        \item \label{item:ocbase} We have either $\pi_B(x') \in B_1$ or $\pi_B(x') \in B_A$ for the split $A = \{ S_Q(x) , [2d] \setminus S_Q(x) \}$ for some $Q \in \pone$.
        \item \label{item:ocjoint} For each split $A = \{S_Q(x) , [2d] \setminus S_Q(x) \}$, $\pi_B^{-1}(B_A) \cap \ol{G \cdot x}_{\mathrm{red}}$ is a single two-dimensional $G$-orbit. In particular, if $f_+ f_- \neq 0$, we have
        \[ \begin{cases}
            f_+' = 0 & (a_Q(x) - b_Q(x) - 1 > 0) \\
            f_+'f_-' \neq 0 & (a_Q(x) - b_Q(x) - 1 = 0) \\
            f_-' = 0 & (a_Q(x) - b_Q(x) - 1 < 0).
        \end{cases} \]
    \end{enumerate}
\end{corollary}
\begin{proof}
    \cref{item:ocbase} This is a standard fact for the pushforward $(\pi_B)_*$.
    \cref{item:ocjoint} This follows from the direct computation by taking limits. The resulting orbit maps onto the two-dimensional orbit $B_A$ and is a proper boundary orbit of the three-dimensional orbit $G\cdot x$, so it is two-dimensional.
\end{proof}

\subsection{Computation of the cycle class of orbits}

Let $\xi := c_1(\oshf_X(1))$ for the $\pone$-bundle structure $\pi_B : X \to B$. From the projective bundle formula, the Chow ring is given by
\eqnsl{
A^*(X) & = A^*(B)[\xi ] / (\xi^2 - c_1(\oshf_B(D_+)\oplus\oshf_B(D_-))\xi + c_2(\oshf_B(D_+)\oplus\oshf_B(D_-))) \nonumber \\
& = \zahl [H_1,\ldots , H_{2d}, \xi] / (H_i^2 , (\xi - D_+) (\xi - D_-)). \label{eq:strAX}
}
Here we also remark that $\xi = \Pi^*\oshf_{\crat_d} (1)$.

\begin{lemma} \label{lem:IntersecNum} Let $x = (P_1, \ldots , P_{2d}, [f_+ : f_-] ) \in X$ have a three-dimensional orbit. Thus its orbit cycle is
\[ \ol{G\cdot x}=(\#\stab_G(x))\ol{G\cdot x}_{\mathrm{red}}. \]
For any $i \in [1,2d]$, put $c_i := c_{P_i}(x)$.

Then we have
\eqnsl{
H_i H_j H_k |_{\ol{G\cdot x}} &= \begin{cases}
    1 & (\# \{ P_i,P_j,P_k \} = 3), \\ 
    0 & (\# \{ P_i,P_j,P_k \} < 3).
\end{cases} \label{eq:hihjhk} \\
\xi H_i H_j |_{\ol{G\cdot x}} &= \begin{cases}
    0 & (P_i = P_j), \\
    (d+1) - c_i - c_j & (P_i \neq P_j). \\
\end{cases} \label{eq:xihihj}
}
\end{lemma}
\begin{proof}
    The former equation follows from the fact that $\pglt$ is 3-transitive on $\pone$.

    These intersections are computed on the orbit parameterization, whose pushforward is the orbit cycle $\ol{G\cdot x}$. For the latter equation, again by transitivity, we can assume that $H_i = 0$ and $H_j = \infty$. We have $\ol{G \cdot x}_{\mathrm{red}} \cap  H_i \cap H_j = \emptyset$ if $P_i = P_j$. Otherwise the pullback of $H_i \cap H_j \cap \ol{G\cdot x}$ to the orbit parameterization is the one-parameter family
    \[ \{t\cdot x\mid t\in\gmult\}. \]
    Here $\gmult$ acts on $x$ as the diagonal matrices in $\pglt$, that is, 
    \eqns{& t \cdot [ [ a_{+,0} , a_{+,1}, \ldots, a_{+,d+1}], [a_{-,1}, a_{-,2}, \ldots , a_{-,d} ] ] \\ & = [ [ a_{+,0}, ta_{+,1}, \ldots , t^{d+1}a_{+,d+1} ] , [ ta_{-,1}, t^2a_{-,2}, \ldots , t^da_{-,d} ] ].
    }
    An arbitrary member $\xi \in H^0(\oshf_{\ol{\mrat}_d}(1))$ is given by 
    \eqns{ & \xi ( [f_+ , f_-] ) = \sum_i \xi_{+,i} a_{+,i} + \sum_i \xi_{-,i} a_{-,i} \text{ for some } \xi_{\pm,i} \in k.
    }
    For a generic $\xi$, the hyperplane $(\xi  = 0)$ does not intersect at $t = 0,\infty$. Therefore, $ (\xi = 0 ) \cap \ol{\gmult \cdot x}$ coincides with an equation of $t$ of degree $(d+1) - c_i - c_j$. This shows the assertion.
\end{proof}
The set $\{ H_iH_jH_k \}_{i<j<k} \cup \{ H_iH_j\xi \}_{i<j}$ is a basis of $A^3(X)$. Under the Poincar\'{e} intersection form
\[ A^3 (X) \otimes A^{2d-2}(X) \to A^{2d+1}(X) = \zahl \cdot [\text{pt}] \xto{[\fct{pt}]^\vee : [\fct{pt}] \mapsto 1}\zahl ,\ (Z,Z') \mapsto [\fct{pt}]^\vee ZZ', \]
we denote its dual basis in $A^{2d-2}(X)$ by $\{ (H_iH_jH_k)^\vee \} \cup \{ (H_iH_j\xi )^\vee \}$.
\begin{corollary} \label{cor:gendegOrbit}
    The cycle class $\delta (G,X)\in A^{2d-2}(X)$ of a generic orbit cycle for $X = \pmcrat_d$ is given by
    \eqnl{ \delta (G,X) = \sum_{1 \leq i < j <k \leq 2d}( H_iH_jH_k )^{\vee} + \sum_{1 \leq i < j \leq 2d}(d+1)( H_iH_j\xi )^{\vee} - \sum_{\substack{1 \leq i \leq d+1 \\ 1 \leq j \leq 2d \\ i \neq j}}(H_iH_j\xi)^\vee.  \label{eq:gendegX}}
    The generic orbit-cycle locus of $X$ is given by
    \eqnslg{  U_X & :=  \{ x \in X \mid \delta(G,x) = \delta (G,X) \} \\
    & = \{ (P_1,\ldots , P_{2d} , [f_+ : f_-]) \in X \mid  P_i \neq P_j\ (\forall i \neq j), \ f_+ \neq 0  \} . }{\label{eq:gendeglocU}}
\end{corollary}
\begin{proof} The assertion follows from direct computation using \Cref{lem:IntersecNum}.
\end{proof}

\section{Stratification of Chow Quotients}
\label{sec:stratification}

\subsection{Review of $\ol{M}_{0,n}$}
In this subsection, we briefly recall the Chow-quotient construction of the moduli space of stable rational curves 
\[ \ol{M}_{0,n} := (\pone )^n \sslash_{Ch} \pglt   \]
by \cite{kapranov1993ChowQuotGrassI}.
The space $\ol{M}_{0,n}$ is a canonical compactification of the space of the point configurations $M_{0,n} = \conf_n\pone \slash \pglt$, where $\conf_n\pone := (\pone )^n \setminus \bigcup_{1\leq i<j\leq n} \Delta_{ij}$.
An $n$-pointed stable rational curve is defined as a connected curve of arithmetic genus $0$ with $n$ nonsingular marked points, where its singularities are locally isomorphic to $xy = 0$ and whose automorphism group preserving marked points is trivial. By taking the dual graph, it is equivalent to the datum $(\Gamma , \{ C_v \}_{v \in V(\Gamma )} )$: where 
\begin{itemize}
    \item $\Gamma$ is a tree graph with leaf-marking $p : \{ 1,\ldots , n \} \to \Gamma$ with no vertex of valence 2 and 
    \item $C : \{ v : \val (v)  \geq 3 \} \ni v \mapsto C_v \in M_{0,\val (v)}$ is the data of point configurations.
\end{itemize}
On the Chow-quotient construction, the point configuration $C_v$ corresponds to a 3-dimensional orbit cycle of $( \pone )^n$ by the vertex decomposition of the leaf marking of the tree graph of the dual, that is, 
\eqns{ & M_{0,\val (v)} \ni C_v = [G \cdot (P_{e})_{e \in E_v}] \leftrightarrow Z_v = \ol{G \cdot (P_{p_v(1)}, \ldots , P_{p_v(n)})} \subset (\pone )^n, \\
& \text{ where } E_v := \{ e \in E(\Gamma ) \mid v \in e \}
}
and $p_v$ is the vertex contraction. Here the tree graph with leaf-marking $( \Gamma , p)$ determines the decomposition of the cycle class. In the Chow ring
\[ A^*((\pone)^n)\simeq\zahl[H_1,\ldots,H_n]/(H_i^2), \]
the class of the three-dimensional orbit cycle is
\[ [Z_v]
=\sum_{\substack{i<j<k: \\ \# p_v(\{i,j,k\})=3}}
(H_iH_jH_k)^\vee
=\sum_{\substack{i<j<k: \\ \# p_v(\{i,j,k\})=3}}
\prod_{\ell\notin\{i,j,k\}}H_\ell
\quad\text{in }A^{n-3}((\pone)^n). \]
Thus the nonempty strata of the decomposition stratification of $\ol{M}_{0,n}$ are indexed by the trees with leaf-marking by $[n]$ and no vertex of valence two.

As a point $C = [(\Gamma , \{ C_v \}) ] \in \ol{M}_{0,n}$, there is another way to obtain the dual graph. Choose a finite valued extension $K$ of $k((t))$. Since $M_{0,n}$ is dense open in $\ol{M}_{0,n}$, by \Cref{prop:valucriterion} and after enlarging $K$ if necessary, we can take $[C_t]\in M_{0,n}(K)$ with $\spe[C_t]=C$ and a representative $C_t=(P_{1,t},\ldots,P_{n,t})\in\pone(K)^n$. Under the natural embedding of classical points
\[ \iota:\pone(\ol K)\inc\ponean_K, \]
let $T$ be the smallest sub $\real$-tree of $\proj^{1,an}_K$ including $\iota (\{ P_{i,t} \} )$. By forgetting the metric of $T$, we obtain a leaf-marked tree $( \Gamma (T) , p )$. Every branching vertex $\zeta$ of $T$ of valence at least three is of type II, and the reduction map 
\[ \fct{red}_\zeta : \ponean \setminus \{ \zeta \} \to \cc (\ponean \setminus \{ \zeta \} ) \simeq \pone_{k}\]
at each such vertex determines the specialized configuration $\zeta \mapsto C_\zeta \in M_{0,\val(\zeta)}$.

\begin{example} \label{ex:ConfigBerkovich}
For the field $K = k((t))$, we take a point 
\[ x = (P_1, P_2,P_3,P_4,P_5,P_6 ) = (\infty , t^{-2},-1,0,1,1+t^3 ) \in \conf_6\pone_K\]
and let $[C_t]$ be the image of $x$ in $M_{0,6}(K)$. We compute the point $\spe [C_t] \in \ol{M}_{0,6}(k)$. The smallest subtree of $\ponean_K$ which includes the half lines of the limit points $P_i$ has the type II points $O(t^{-2}), O(1), 1+O(t^3)$ as its branch points, where we write $f + O(t^N)$ for the type II point corresponding to the non-Archimedean disc $D(f, |t^N| )$.
\usetikzlibrary{calc,positioning}

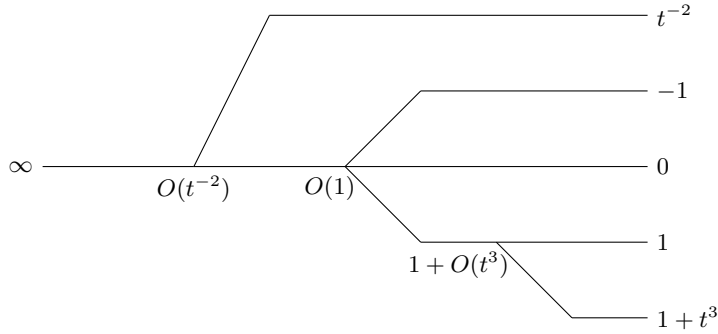
\begin{tikzpicture}[>=latex,
    every node/.style={font=\small}]

\coordinate (A) at (0,0);
\coordinate (B) at (2,0);
\coordinate (C) at (3,1);
\coordinate (D) at (3,-1);
\coordinate (E) at (1,2);
\coordinate (F) at (4,-1);
\coordinate (G) at (5,-2);

\draw (-2,0)--(A);
\node[left] at (-2,0) {$\infty$};

\draw (A)--(B);

\draw (A)--(E);
\draw (B)--(C);
\draw (B)--(D);
\draw (F)--(G);

\draw (E)--++(5,0);
\draw (C)--++(3,0);
\draw (B)--++(4,0);
\draw (D)--++(3,0);
\draw (G)--++(1,0);

\node[below] at (A) {$O(t^{-2})$};
\node[below] at ($(B)+(-0.2,0)$) {$O(1)$};
\node[below] at ($(D)+(0.5,0)$) {$1+O(t^{3})$};

\node[right] at ($(E)+(5,0)$) {$t^{-2}$};
\node[right] at ($(C)+(3,0)$) {$-1$};
\node[right] at ($(B)+(4,0)$) {$0$};
\node[right] at ($(D)+(3,0)$) {$1$};
\node[right] at ($(G)+(1,0)$) {$1+t^3$};
\end{tikzpicture}

Each branch point corresponds to a sphere in the tree-of-spheres formulation and (relative) coordinates of each point on the sphere are determined by the value of the reduction.

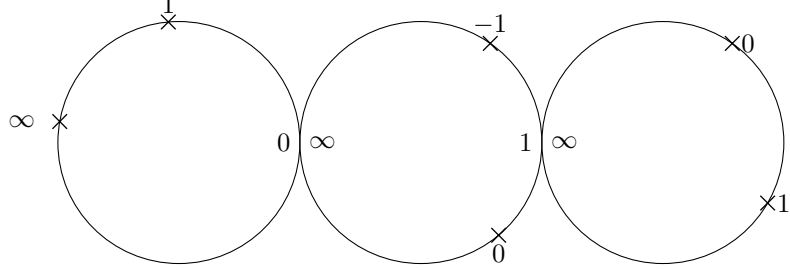
\usetikzlibrary{shapes.misc}
\begin{tikzpicture}[
    scale=1,
    pt/.style={
        cross out,
        draw,
        minimum size=5pt,
        inner sep=0pt,
        line width=.5pt
    }
]

\def\r{1.6}

\draw (0,0) circle (\r);
\draw (2*\r,0) circle (\r);
\draw (4*\r,0) circle (\r);

\node[pt] (Linf) at ({\r*cos(170)},{\r*sin(170)}) {};
\node[left=2mm] at (Linf) {$\infty$};

\node[pt] (L1) at ({\r*cos(95)},{\r*sin(95)}) {};
\node[above] at (L1) {$1$};

\node[left] at (\r,0) {$0$};

\node[right] at (2*\r-\r,0) {$\infty$};

\node[pt] (M1) at ({2*\r+\r*cos(55)},{\r*sin(55)}) {};
\node[above] at (M1) {$-1$};

\node[pt] (M2) at ({2*\r+\r*cos(-50)},{\r*sin(-50)}) {};
\node[below] at (M2) {$0$};

\node[left] at (3*\r,0) {$1$};

\node[right] at (4*\r-\r,0) {$\infty$};

\node[pt] (R1) at ({4*\r+\r*cos(55)},{\r*sin(55)}) {};
\node[right] at (R1) {$0$};

\node[pt] (R2) at ({4*\r+\r*cos(-30)},{\r*sin(-30)}) {};
\node[right] at (R2) {$1$};

\end{tikzpicture}

Sometimes it is convenient to write each sphere $\pone_k$ as a segment.

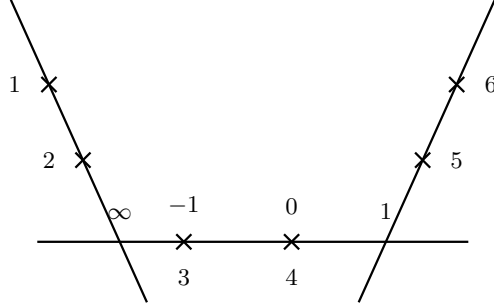
\begin{tikzpicture}[
  x=1cm,y=1cm,
  main line/.style={line width=.85pt},
  tick/.style={line width=.85pt},
  every node/.style={font=\small}
]

\draw[main line]
  (-3.2,3.2) -- coordinate[pos=.28] (L1)
                 coordinate[pos=.53] (L2) (-1.4,-.8);

\draw[main line]
  (-2.85,0) -- coordinate[pos=.34] (H3)
                coordinate[pos=.59] (H4) (2.85,0);

\draw[main line]
  (1.4,-.8) -- coordinate[pos=.47] (R5)
                coordinate[pos=.72] (R6) (3.2,3.2);

\newcommand{\xmark}[1]{%
  \draw[tick] ($(#1)+(-.10,-.10)$) -- ($(#1)+(.10,.10)$);
  \draw[tick] ($(#1)+(-.10,.10)$)  -- ($(#1)+(.10,-.10)$);
}
\xmark{L1}
\xmark{L2}
\xmark{H3}
\xmark{H4}
\xmark{R5}
\xmark{R6}

\node[left=7pt]  at (L1) {$1$};
\node[left=7pt]  at (L2) {$2$};
\node[below=7pt] at (H3) {$3$};
\node[below=7pt] at (H4) {$4$};
\node[right=7pt] at (R5) {$5$};
\node[right=7pt] at (R6) {$6$};

\node[above=5pt] at (-1.76,0) {$\infty$};
\node[above=7pt] at (H3)     {$-1$};
\node[above=7pt] at (H4)     {$0$};
\node[above=5pt] at (1.76,0) {$1$};

\end{tikzpicture}

\end{example}

\subsection{Strata Trees and the Statement of Structure Theorem}

We will see an analogue of this method on $\pmcrat_d$. For $\pmcrat_d$, the vertices of the dual graphs are decorated by sign symbols we name contour.

\begin{definition}
    Let $(\Gamma,p)$ be a leaf-marked tree with $p:I\xto{\sim}L(\Gamma)$. A \emph{leaf-weight} is a map $c:I\to\{+1,-1\}$; thus the weight is attached to a label, rather than evaluated on an arbitrary vertex. We call $(\Gamma,p,c)$ a \emph{weighted leaf-marked tree}.

    For an edge $e=\{v,v'\}$, let $p_e:I\to\{v,v'\}$ be the edge-contracted marking and put
    \[ W_e(v):=\sum_{i\in p_e^{-1}(v)}c(i),\qquad W_e(v'):=\sum_{i\in p_e^{-1}(v')}c(i). \]
    We write $v\sim_e v'$ if $W_e(v)=W_e(v')$, and $v<_e v'$ if $W_e(v)<W_e(v')$. This is an edge-wise relation; the notation does not identify the two vertices of the edge.

    A \emph{contour} is a map $\sigma:V(\Gamma)\to\{-,0,+\}$ for which there are positive edge lengths $\ell$ and a continuous piecewise-linear function
    $h:T(\Gamma,\ell)\to\real$ such that $\sigma(v)=\sgn h(v)$ and, on the coordinate $x$ oriented from $v$ to $v'$,
    \eqnl{ \partial_x h=\frac{W_e(v')-W_e(v)}{2}. \label{eq:weightorder} }
    We require the boundary of every connected component of $h^{-1}(0)$ to consist of vertices. Consequently, if $v\sim_e v'$, then $\sigma(v)=\sigma(v')$; if $v<_e v'$, then
    $(\sigma(v),\sigma(v'))\in\{(-,-),(-,0),(0,+),(+,+)\}$.
\end{definition}
\begin{remark}
    Unless $\sigma^{-1}(0)=\emptyset$, the contour on a weighted tree is determined by the set $\sigma^{-1}(0)$. If $\sigma^{-1}(0)=\emptyset$, the contour is constant, and the zero set alone does not distinguish $\sigma\equiv+$ from $\sigma\equiv-$. For a strata tree on $\pmcrat_d$ as defined below, condition~\textup{(b)} forces $\sigma\equiv-$ in this exceptional case.
\end{remark}

\begin{definition} \label{def:stratatree}
    \emph{A strata tree on }$\pmcrat_d$ is a tuple $( \Gamma ,\sigma )$ of 
    \begin{enumerate}
        \item \label[cond]{cond:leafwtmarktree} a weighted leaf-marked tree $T = (\Gamma , p,c)$ with a bijection $p:[2d]\xto{\sim}L(\Gamma)$ and
        \item \label[cond]{cond:contourtree} a contour $\sigma$ of $T$
    \end{enumerate}
    which satisfies the following conditions:
    \begin{enumerate}
        \item[(a)] $c(i) = \begin{cases} +1 & (i \leq d+1), \\ -1 & (i \geq d+2), \end{cases}$
        \item[(b)] $\sigma(p(i))=-$ for $i\geq d+2$, and
        \item[(c)] if $\val(v)=2$, then $\sigma(v)=0$ and its two adjacent vertices have signs $-$ and $+$, respectively.
    \end{enumerate}
\end{definition}

The strata tree records the decomposition of the degree of the Chow cycle into the degrees of its irreducible components. The tree loci refine the decomposition stratification, but distinct strata trees may give the same decomposition of cycle classes.

\begin{corollary} \label{cor:stratavariety}
    For a strata tree $(\Gamma,\sigma)$ on $\pmcrat_d$, let $(X_d)_{\Gamma,\sigma}$ be the locus of points whose Chow data have underlying strata tree $(\Gamma,\sigma)$. These tree strata give a refinement
    \[ X_d:=\pmcrat_d\chquo\pglt
    =\bigsqcup_{(\Gamma,\sigma)}(X_d)_{\Gamma,\sigma} \]
    of the decomposition stratification of $X_d$; distinct tree strata may be contained in the same decomposition stratum. Put $V_{\mathrm{st}}(\Gamma):=\{v\in V(\Gamma)\setminus L(\Gamma)\mid\val(v)\geq3\}$. Then
    \eqnl{ (X_d)_{\Gamma,\sigma}\simeq
    \gmult^{\alpha (\Gamma ,\sigma )} \times \prod_{v \in V_{\mathrm{st}}(\Gamma)}M_{0,\val (v)}, \label{eq:stratavariety} }
    where $\alpha (\Gamma , \sigma )$ is the number of connected components of the subgraph induced by
    $\{v\in V_{\mathrm{st}}(\Gamma)\mid\sigma(v)=0\}$.
\end{corollary}

\begin{definition} \label{def:Chowdatum}
    A \emph{Chow datum} on $\pmcrat_d$ is a tuple $(\Gamma , \sigma , Z )$ of a strata tree $(\Gamma , \sigma )$ of $\pmcrat_d$ and $Z: V(\Gamma ) \setminus L(\Gamma ) \to Z_3(\pmcrat_d ) $ such that
    \begin{enumerate}
        \item \label[cond]{cond:edgeorbit} for every edge $e=\{v,v'\}\in E(\Gamma)$ with $v,v'\notin L(\Gamma)$, put $A_e:=\fib{p_e}$. There exists a two-dimensional $G$-orbit $O_e\subset\pmcrat_d$ such that
        \[ O_e\subset\supp Z(v)\cap\supp Z(v'),
        \qquad \pi_B(O_e)=B_{A_e}; \]
        \item \label[cond]{cond:eachcycle} for each $v \in V(\Gamma ) \setminus L(\Gamma ) $, the cycle $Z(v) \in Z_3(\pmcrat_d)$ is the orbit cycle
        \[ Z(v)=\ol{G\cdot x_v}=(\#\stab_G(x_v))\ol{G\cdot x_v}_{\mathrm{red}} \]
        for some $x_v = (b,[f_+ : f_-]) \in \pmcrat_d$ with finite stabilizer, such that
        \begin{enumerate}
            \item $\fib{b} = \fib{p_v}$ and 
            \item $\begin{cases}
            f_+ = 0 & (\sigma (v ) = +) \\
            f_+f_- \neq 0 & (\sigma (v) = 0)\\
            f_- = 0 & (\sigma (v) = -)
        \end{cases}$.
        \end{enumerate}
    \end{enumerate}
\end{definition}

A main purpose of this section is to show the following correspondence:

\begin{theorem} \label{thm:Ydatum}
    There is a bijection
    \[ \fct{ChowData}:X_d:=\pmcrat_d\sslash_{Ch}G
    \xto{\sim}\{\text{Chow data on }\pmcrat_d\} \]
    such that 
    \eqnl{ y=\sum_{v\in V(\Gamma)\setminus L(\Gamma)}Z(v) \text{ for }\fct{ChowData}(y)=(\Gamma,\sigma,Z). \label{eq:ChowData} }
\end{theorem}
We first check that the sum of the rational-equivalence cycle classes in a Chow datum coincides with the generic orbit-cycle class.
\begin{lemma} \label{lem:Chowdatumclass}
    Let $(\Gamma,\sigma,Z)$ be a Chow datum on $X=\pmcrat_d$. Then
    \[ \sum_{v\in V(\Gamma)\setminus L(\Gamma)}[Z(v)]
    =\delta(G,X) \quad\text{in }A^{2d-2}(X). \]
\end{lemma}
\begin{proof}
    We compare intersections with the basis
    $\{H_iH_jH_k\}_{i<j<k}\cup\{H_iH_j\xi\}_{i<j}$ of $A^3(X)$ used in \Cref{cor:gendegOrbit}.

    Fix distinct $i,j,k$. By \Cref{lem:IntersecNum}, the intersection of $Z(v)$ with $H_iH_jH_k$ is $1$ precisely when $p_v(i),p_v(j),p_v(k)$ are distinct, and is $0$ otherwise. There is a unique internal vertex of $\Gamma$ at which the three paths joining $p(i),p(j),p(k)$ meet. Hence
    \[ H_iH_jH_k\bigm|_{\sum_v Z(v)}=1. \]

    Now fix distinct $i,j$, let $\mathcal{P}_{ij}$ be the internal vertices on the path from $p(i)$ to $p(j)$, and orient this path from $p(i)$ to $p(j)$. For $v\in\mathcal{P}_{ij}$, choose a representative
    $x_v=((P_{1,v},\ldots,P_{2d,v}),[f_{+,v}:f_{-,v}])$ as in \cref{cond:eachcycle}, and put $c_{i,v}:=c_{P_{i,v}}(x_v)$ and $c_{j,v}:=c_{P_{j,v}}(x_v)$. After normalizing $P_{i,v}=0$ and $P_{j,v}=\infty$, the Newton polytope of the diagonal-torus orbit of $[f_{+,v}:f_{-,v}]$ is
    \[ I_v=[c_{i,v},d+1-c_{j,v}]. \]
    Indeed, these are the smallest and largest nonzero torus weights in $V_{d+1}\oplus V_{d-1}$. Thus, with orbit-cycle multiplicity, the lattice length of $I_v$ is $\xi H_iH_j|_{Z(v)}$.

    Let $v,v'$ be consecutive vertices on the oriented path. By a casewise computation of signs of vertices and numbers of leaves on each side $v$ and $v'$, the definition of $c_P(x)$ gives
    \[ c_{i,v'}=d+1-c_{j,v}. \]
    Hence the intervals $I_v$ concatenate along the path. At the two ends their outer endpoints are $\epsilon_i$ and $d+1-\epsilon_j$, where $\epsilon_m=1$ for $m\leq d+1$ and $0$ otherwise. Vertices off the path contribute zero by \Cref{lem:IntersecNum}. Therefore
    \[ \xi H_iH_j\bigm|_{\sum_vZ(v)}
    =\sum_{v\in\mathcal{P}_{ij}}\fct{length}(I_v)
    =d+1-\epsilon_i-\epsilon_j. \]
    These are exactly the intersections of the generic orbit-cycle class in \cref{eq:gendegX} with the chosen basis, which proves the equality of rational-equivalence classes.
\end{proof}

\subsection{Specializations and Orbits} \label{sec:specializations}
In the analysis of the Chow quotient, constraints among orbits under different specializations are crucial. We fix a finite valued extension $K$ of $k((t))$, and extend its valuation to $\ol K$. As above, analytic points are taken over $\widehat{\ol K}$, whereas any finite collection of algebraic marked points and coordinate changes is defined over a finite extension of $K$.
Write $\zeta=D(0,1)$ for the Gauss point and $\red_\zeta$ for its reduction map to $\pone(k)$. We choose coordinate changes toward the Gauss point: for $\gamma\in\pglt(\ol K)$ put
\[ \zeta_\gamma:=\gamma^{-1}\zeta,\qquad
\red_\gamma:=\red_\zeta\circ\gamma. \]
For a marked map $x$, its specialization in these coordinates is $\spe(\gamma\cdot x)$. The notation $\spe_\zeta$ denotes coefficient specialization in Gauss coordinates. Order differences after a coordinate change are computed using an $\slt$-lift; changing the lift does not change the coefficient valuations.

\begin{definition}For the additive valuation $v_K : K \to \real \cup \{ +\infty \}$ and a point $[f] = [ f_{+} , f_{-}] \in \crat_d ( K )$, \emph{the order difference} is the element of $\real\cup\{\pm\infty\}$ defined by
\eqns{
\ord & \fct{diff} [f] := \min_i v_K(a_{i,+}) - \min_i v_K(a_{i,-}) \text{, where } \widetilde{f}_{\pm} = \sum_i a_{i,\pm} x_0^{d\pm1 - i} x_1^i \\
& \text{for any affine lifting }(\widetilde{f}_+ , \widetilde{f}_-) \in (V_{d+1} \oplus V_{d-1})(K) \text{ of }[f_+ , f_- ].
}
Here a zero summand has minimum coefficient valuation $+\infty$, and we use the conventions $a-(+\infty)=-\infty$ and $(+\infty)-a=+\infty$ for $a\in\real$.
\end{definition}

\begin{lemma} \label{lem:Orddiff}
    Let $[f] = [ f_{+} , f_{-}] \in \mrat_d ( \fsfld )$ and put \[ \spe [f] =:  [\ol{f}_+ , \ol{f}_{-}] \in \ol{\mrat}_d  (k). \] Then we have 
        \eqns{ 
        & \ol{f}_{+}= 0 \text{ if and only if }  \ord \fct{diff} [f] > 0 \text{, and}\\ 
        & \ol{f}_- = 0 \text{ if and only if }\ord \fct{diff} [f] < 0. }
\end{lemma}
\begin{proof}
This is immediate from the definition of the specialization $\spe [f]$ and the order function.
\end{proof}
\begin{lemma} \label{lem:Xmaxdimorbit}
 Let $x=(P_1,\ldots,P_{2d},[f_+:f_-])\in\pmcrat_d(K)$. For an arbitrary $\gamma\in\pglt(\ol K)$, let \[ \fct{red}_\gamma:\ponean_K\setminus\{\zeta_\gamma\}\to T_{\pone}\zeta_\gamma\simeq\pone(k) \] be the reduction map in the coordinates specified above.
    We put \[ S_\gamma := \{ Q \in \pone_k \mid Q = \red_{\gamma} ( P_i) \text{ for some }i \in \{1,\ldots, 2d \} \} \]
    and for any $Q \in S_{\gamma}$, 
    \[ S_Q := \{ i \mid \fct{red}_{\gamma}(P_i) = Q \}. \]
    If $\xi:=\gamma^{-1}\cdot\zeta$ and $\dim \ol{\pglt\cdot\spe(\gamma\cdot x)}=3$, the orbit cycle
    \[ Z_\xi(x):=\ol{\pglt\cdot\spe(\gamma\cdot x)} \]
    depends only on $\xi$ and $x$, and not on the choice of $\gamma$.
    Moreover,
    \[ \dim \ol{ \pglt \cdot \spe (\gamma \cdot x) } = 3 \]
    if and only if
    \begin{enumerate}
        \item $\# S_\gamma \geq 3$, or
        \item $\# S_\gamma = 2,\ \ord \diff (\gamma \cdot [f_+, f_-]) = 0$ and for the points $Q, Q'$ in $S_{\gamma}$, 
        \[ \# ( S_{Q} \cap [1,d+1] ) - \# ( S_{Q} \cap [d+2, 2d] ) \neq  \# ( S_{Q'} \cap [1,d+1] ) - \# ( S_{Q'} \cap [d+2, 2d] ).  \]
    \end{enumerate}
\end{lemma}
\begin{proof}
    If $\gamma^{-1}\zeta=\gamma'^{-1}\zeta$, then $\gamma'\gamma^{-1}\in\pglt(\oring_{\ol K})$, whose reduction carries $\spe(\gamma\cdot x)$ to $\spe(\gamma'\cdot x)$. Their orbit cycles, including stabilizer multiplicities, therefore coincide, so $Z_\xi(x)$ is well-defined.

    After a finite extension of $K$ and replacing $x$ by $\gamma\cdot x$, it is enough to treat $\gamma=\fct{id}$. Then the assertion follows from \Cref{prop:orbdecomp} and \Cref{lem:Orddiff}.
\end{proof}

\begin{lemma} \label{lem:orddifftree}
    Let $x_K=(P_1,\ldots,P_{2d},[f_+:f_-])\in\pmcrat_d(K)$ with $f_+f_-\neq0$, and let $T_x$ be the combinatorially finite $\real$-tree spanned by the marked points. For any type~II point $\xi=\zeta_\gamma\in T_x$, put
    \[ h_x(\xi):=\ord\fct{diff}(\gamma\cdot[f_+:f_-]). \]
    \begin{enumerate}
        \item \label{item:hxwelldef} The function $h_x$ is independent of the choice of $\gamma$ with $\gamma^{-1}\cdot\zeta=\xi$.
        \item \label{item:hslope} Take a segment of $T_x$ oriented from $v$ to $v'$. If its interior contains no branch point or marked retraction, then
    \eqnl{ \partial_{v\to v'}h_x=\frac{W_e(v')-W_e(v)}{2}. \label{eq:slope}}
        Thus $h_x$ extends to a continuous piecewise-linear function on the nonclassical part of $T_x$.
        \item \label{item:htree} There are at most finitely many $\xi \in \ponean_K$ such that $\dim \ol{\pglt\cdot\spe(\gamma\cdot x_K)}=3$ and $\gamma^{-1}\cdot\zeta=\xi$. Moreover, every such point $\xi$ is contained in $T_x$.
    \end{enumerate}
\end{lemma}
\begin{proof}
    Let us take a factorization of an affine lifting 
    \[ \widetilde{f}_+=c_+\prod_{i=1}^{d+1}[P_i,z],\qquad \widetilde{f}_-=c_-\prod_{i=d+2}^{2d}[P_i,z]. \]

    \cref{item:hxwelldef} The valuation of each linear factor is invariant under the action of $\pglt(\oring_K)$.

    \cref{item:hslope} With an $\slt$-normalized coordinate change, the minimum coefficient valuation of each transformed linear factor is affine with slope $+1/2$ or $-1/2$ along a segment of $T_x$, according to the component containing its marked endpoint. Subtracting the two sums gives \cref{eq:slope} with the metric of Section~2.3.

    \cref{item:htree} By \Cref{lem:Xmaxdimorbit}, a point $\xi = \gamma^{-1}\cdot\zeta$ with $\dim \ol{\pglt\cdot\spe(\gamma\cdot x_K)}=3$ is contained in $T_x$. Moreover, if $\# S_\gamma = 2$, then the value
    \[
    \# ( S_{Q} \cap [1,d+1] ) - \# ( S_{Q} \cap [d+2, 2d] ) - ( \# ( S_{Q'} \cap [1,d+1] ) - \# ( S_{Q'} \cap [d+2, 2d] )) \]
    is twice the slope of $h_x$ along the segment oriented toward $Q$. Since this slope is nonzero, there are at most finitely many such points $\xi$ with $h_x(\xi) = 0$.
\end{proof}

\subsection{The construction of the map ChowData}
We first construct the map $\fct{ChowData}$.
\begin{proposition} \label{prop:YdatumInj}
    There is an injection
    \[ \fct{ChowData}:X_d:=\pmcrat_d\sslash_{Ch}G
    \to\{\text{Chow data on }\pmcrat_d\} \]
    such that, if $\fct{ChowData}(y)=(\Gamma,\sigma,Z)$, then
    \eqnl{ y=\sum_{v\in V(\Gamma)\setminus L(\Gamma)}Z(v). }
\end{proposition}
\begin{proof}
    We use the specialization results of \Cref{lem:orddifftree,lem:Orddiff,lem:Xmaxdimorbit}.

    Let $y\in X_d$. After a finite valued-field extension, the valuative criterion for properness gives a point
    \[ x_K=(P_1,\ldots,P_{2d},[f_+:f_-])\in U_X(K),\qquad f_+f_-\neq0, \]
    whose orbit cycle specializes to $y$. Let $T_x$ and $h_x$ be as in \Cref{lem:orddifftree}. By \Cref{lem:orddifftree}\cref{item:htree}, the set \[ T_S := \{ \xi \in \ponean_K \mid \xi = \gamma^{-1}\cdot\zeta,\ \dim \ol{G \cdot  \spe (\gamma \cdot x_K )} = 3 \} \]
    is finite, and each $\xi\in T_S$ determines the orbit cycle $Z_\xi(x_K)$. Suppressing redundant vertices gives a finite leaf-marked tree $\Gamma_y$.

    Let $\ell_i\in L(\Gamma_y)$ be the leaf corresponding to the marked analytic point $P_i$, and put $p_y(i)=\ell_i$. Let $c_y(i)=+1$ for $i\leq d+1$ and $c_y(i)=-1$ for $i\geq d+2$, and put $\sigma_y(v)=\sgn h_x(v)$ at internal vertices, using the limiting sign along the incident ray at each leaf. The slope formula of \Cref{lem:orddifftree} and the specialization criterion of \Cref{lem:Orddiff} show that $(\Gamma_y,\sigma_y)$ is a strata tree. For an internal vertex $v$, viewed as its corresponding point of $T_S$, put
    \[ Z_y(v):=Z_v(x_K). \]
    The vertex orbit closures have distinct marked-point partitions. By \Cref{prop:stabilizermultiplicity}, their contributions give
    \[ y \geq \sum_{v\in V(\Gamma_y)\setminus L(\Gamma_y)}Z_y(v). \]
    For every internal edge $e=\{v,v'\}$ of $\Gamma_y$, specialization at an ordinary point of the corresponding analytic segment gives a two-dimensional orbit $O_e$. Its marked-point partition is $A_e=\fib{p_e}$, and it is contained in both $\supp Z_y(v)$ and $\supp Z_y(v')$. Thus \cref{cond:edgeorbit} holds, while the construction gives \cref{cond:eachcycle}; hence $(\Gamma_y,\sigma_y,Z_y)$ is a Chow datum. By \Cref{lem:Chowdatumclass}, the cycle on the right has class $\delta(G,X)$, which is also the class of $y$. Their effective difference therefore has degree zero with respect to an ample divisor and must vanish, so the displayed inequality is an equality.
    Since an effective cycle has a unique decomposition into reduced irreducible components with multiplicities, and the fiber partitions and signs on each component recover the strata tree, this construction defines an injective map $\fct{ChowData}$.
\end{proof}

\subsection{Constant multiplication by twisting} \label{sec:branchtwist}

For surjectivity of the map, a key part is to adjust the aspect constant, the ratio of constants $[c_+ : c_-]$ of the point $[c_+ f_+ : c_- f_-] \in \crat_d$.

For the diagonal torus $T := \left\{ \begin{bmatrix} p & 0 \\ 0 & q\end{bmatrix} : pq \neq 0 \right\}$ of $\glt$ and its $k$-point $\gamma = \begin{bmatrix} p & 0 \\ 0 & q\end{bmatrix} \in T(k)$, we define \emph{the branch twist} on $K^2$ by 
\[ \tau_\gamma : K^2 \to K^2 : \tau_\gamma ( \beta_0, \beta_1 ) := \begin{cases}
    (q\beta_0, p\beta_1) & (v_K(\beta_0)<v_K(\beta_1)), \\
    (   \beta_0 , \beta_1) & (\text{otherwise}).
\end{cases} \]

The branch twist induces self-maps on $V_{n,K} = K[x_0,x_1]_n,\ \crat_d(K),$ and $\pmcrat_d (K)$ respectively by 
\eqns{ & \tau_{\gamma ,n} :  V_{n,K} \to V_{n,K} & &: c \prod_{i=1}^n [ \beta_i, x] \mapsto c \prod_{i=1}^n [ \tau_\gamma (\beta_i ), x]  \\
 & \tau_{\gamma, \ol{R}} :  \crat_d(K) \to \crat_d(K) & & :[f_+, f_-] \mapsto [\tau_{\gamma, d+1}(f_+), \tau_{\gamma , d-1} (f_-) ]\\
 & \ol{\tau_{\gamma, \pm}} : \pmcrat_d(K) \to \pmcrat_d(K) & & :  (P_1,\ldots , P_{2d}, [f_+,f_-]) \\
 & & & \ \ \mapsto (\tau_\gamma (P_1),\ldots , \tau_{\gamma}(P_{2d}), \tau_{\gamma , \ol{R}} ([f_+,f_-])). }

Put $s:=p/q$ and let
\[ \mu_s:\pone_K\longrightarrow\pone_K,
\qquad [u:v]\longmapsto[u:sv]. \]
The map $\tau_\gamma$ induces a map on $\pone_K$, denoted by $\ol{\tau_\gamma}$, which equals $\mu_s$ on the tangent direction $0$ at the Gauss point $\zeta$ and is the identity on every other tangent direction. Its extension to Berkovich points is defined by the three mutually exclusive cases
\[ \ol{\tau_\gamma}^{an}(\xi)=
\begin{cases}
\zeta & (\xi=\zeta),\\
\mu_s^{an}(\xi) & (\xi\neq\zeta\text{ and }\red_\zeta(\xi)=0),\\
\xi & (\xi\neq\zeta\text{ and }\red_\zeta(\xi)\neq0).
\end{cases} \]
For each type~II point $\xi=D(z_\xi,|a_\xi|)$, choose $z_\xi\in\ol K$ and $a_\xi\in\ol K^\times$, and put
\[ \iota_\xi:=\begin{bmatrix}a_\xi&z_\xi\\0&1\end{bmatrix}
\in\pglt(\ol K),\qquad \iota_\xi(\zeta)=\xi, \]
and take $\iota_\zeta=\id$. These choices can be made over a finite extension of $K$ for any finite collection of type~II points. If $\eta=\mu_s^{an}(\xi)$, we make the compatible choices $z_\eta=sz_\xi$ and $a_\eta=a_\xi$. For $x\in\pone(\ol K)$, write
\[ \red_\xi(x):=\red_\zeta(\iota_\xi^{-1}(x))\in\pone(k). \]
Thus $\red_\xi=\red_\gamma$ for $\gamma=\iota_\xi^{-1}$.

\begin{lemma} \label{lem:ShiftPoint}
For any $\gamma\in T(k)$, $x\in\pone(K)$, and type~II point $\xi\in\ponean_K$, let $y:=\ol{\tau_\gamma}(x)$ and $\eta:=\ol{\tau_\gamma}^{an}(\xi)$. Then
\[ \red_\eta(y)=
\begin{cases}
\mu_s(\red_\xi(x)) & (\xi\neq\zeta\text{ and }\red_\zeta(\xi)=0),\\
\red_\zeta(x) & (\xi=\zeta),\\
\red_\xi(x) & (\xi\neq\zeta\text{ and }\red_\zeta(\xi)\neq0).
\end{cases} \]
\end{lemma}
\begin{proof}
Suppose first that $\xi\neq\zeta$ and $\red_\zeta(\xi)=0$. If $x$ is in the tangent direction $0$ at $\zeta$, then $y=\mu_s(x)$ and the compatible coordinate choices give
\[ \iota_\eta^{-1}\circ\mu_s\circ\iota_\xi=\mu_s. \]
Hence $\red_\eta(y)=\mu_s(\red_\xi(x))$. If $x$ is outside that direction, both $\red_\xi(x)$ and $\red_\eta(y)$ are the direction toward $\zeta$, namely $\infty$, which is fixed by $\mu_s$.

If $\xi=\zeta$, write $x=[u:v]$. When $\red_\zeta(x)=0$, one has $v_K(u)<v_K(v)$ and
\[ y=[qu:pv],\qquad \red_\zeta(y)=[\ol u:0]=\red_\zeta(x). \]
Outside the direction $0$, the branch twist fixes $x$. Thus $\red_\zeta(y)=\red_\zeta(x)$ in all cases.

Finally, suppose that $\xi\neq\zeta$ and $\red_\zeta(\xi)\neq0$. Points modified by the branch twist lie in the direction from $\xi$ toward $\zeta$, and their images remain in that same direction; all other points are fixed. Therefore $\red_\xi(y)=\red_\xi(x)$.
\end{proof}
This lemma implies that under the Chow quotient morphism $\pi_{\ol{M}}  : \conf_n \pone \to  (\pone )^n \chquo G = \ol{M}_{0,n}$, we have
\[ \spe \pi_{\ol{M}} (\ol{\tau_\gamma} (P_1) ,\ldots , \ol{\tau_\gamma}(P_n)) = \spe \pi_{\ol{M}} (P_1,\ldots,P_n). \]

\begin{lemma} \label{lem:ShiftConstMult}
For a point $x = (P_1,\ldots , P_{2d}, [f_+,f_-]) \in \pmcrat_d (K)$, put $a := a_{0}(\spe x)$ and $b := b_{0}(\spe x)$, the numbers of $+$ and $-$ markings specializing to $0$, respectively. Moreover, for any $\gamma = \begin{bmatrix}
    p & 0 \\ 0 & q
\end{bmatrix}\in T(k)$ and any type II point $\xi \in \ponean_K$, put
\[ [\ol{f}_+,\ol{f}_-] := \spe_\zeta\!\left(\iota_\xi^{-1}[f_+,f_-]\right),
\qquad \eta := \ol{\tau_\gamma}^{an}(\xi). \]
For a binary form $f=c\prod_{i=1}^n[\beta_i,x]$, write
\[ \mu_{s,n}(f):=c\prod_{i=1}^n[\mu_s(\beta_i),x]. \]
Put $\ol f_{+,s}:=\mu_{s,d+1}(\ol f_+)$ and $\ol f_{-,s}:=\mu_{s,d-1}(\ol f_-)$. 
Then
\[ \spe_\zeta\!\left(\iota_\eta^{-1}\tau_{\gamma,\ol R}([f_+,f_-])\right)
=\begin{cases}
[q^2p^{a-b-2}\ol f_{+,s}:\ol f_{-,s}]
& (\xi\neq\zeta\text{ and }\red_\zeta(\xi)=0),\\
[q^{a-b}\ol f_+:\ol f_-] & (\xi=\zeta),\\
[q^{a-b}\ol f_+:\ol f_-]
& (\xi\neq\zeta\text{ and }\red_\zeta(\xi)\neq0).
\end{cases} \]
\end{lemma}
\begin{proof}
The formula is immediate for a zero summand. For a nonzero $f\in V_{n,K}$, factor $f=c_0\prod_{i=1}^n[\beta_i,x]$, with each $\beta_i$ normalized so that its two coordinates have minimum valuation zero, and put
\[ c(f):=\#\{i\in[n]\mid v_K(\beta_{i,0})<v_K(\beta_{i,1})\}
=\#\{i\in[n]\mid\red_\zeta(\beta_i)=0\}. \]
Thus $c(f_+)=a$ and $c(f_-)=b$ whenever the corresponding summand is nonzero.

Suppose that $\xi\neq\zeta$ and $\red_\zeta(\xi)=0$. Write $\xi=D(z_\xi,|a_\xi|)$ and $\eta=D(sz_\xi,|a_\xi|)$. For a linear factor $L_\beta=\beta_0x_1-\beta_1x_0$, use the same choice of $a_\xi^{1/2}$ in the $\slt$-lifts of $\iota_\xi^{-1}$ and $\iota_\eta^{-1}$. Direct substitution gives
\[ \iota_\xi^{-1}(L_\beta)
=a_\xi^{1/2}\bigl(\beta_0x_1-a_\xi^{-1}(\beta_1-z_\xi\beta_0)x_0\bigr). \]
The common factor $a_\xi^{n/2}$ on $V_n$ cancels when comparing the pre-twist and post-twist specializations.
After applying the branch twist and using $z_\eta=sz_\xi$, its specialization is
\[ \spe_\zeta\!\left(\iota_\eta^{-1}\tau_\gamma(L_\beta)\right)
=\begin{cases}
q\,\mu_{s,1}\!\left(\spe_\zeta(\iota_\xi^{-1}L_\beta)\right)
& (\red_\zeta(\beta)=0),\\
\dfrac{q}{p}\,\mu_{s,1}\!\left(\spe_\zeta(\iota_\xi^{-1}L_\beta)\right)
& (\red_\zeta(\beta)\neq0).
\end{cases} \]
Multiplication over the factors therefore yields
\eqnl{
\begin{aligned}
\spe_\zeta\!\left(\iota_\eta^{-1}\tau_{\gamma,n}(f)\right)
&=q^{c(f)}\left(\frac qp\right)^{n-c(f)}
\mu_{s,n}\!\left(\spe_\zeta(\iota_\xi^{-1}f)\right)\\
&=q^np^{c(f)-n}\mu_{s,n}\!\left(\spe_\zeta(\iota_\xi^{-1}f)\right).
\end{aligned}
\label{eqn:speconst}}

At the Gauss point, a factor in the direction $0$ acquires the scalar $q$, while every other factor is unchanged. The same scalar calculation holds at a point $\xi\neq\zeta$ in any other tangent direction: every factor in the direction $0$ reduces at $\xi$ to the direction toward $\zeta$ and acquires the scalar $q$, while the remaining factors are fixed. Hence in each of these two cases
\[ \spe_\zeta\!\left(\iota_\xi^{-1}\tau_{\gamma,n}(f)\right)
=q^{c(f)}\spe_\zeta\!\left(\iota_\xi^{-1}f\right). \]

Applying these formulas first with $(n,c(f))=(d+1,a)$ and then with $(n,c(f))=(d-1,b)$, and only then passing to the projective pair, gives the three asserted formulas.
\end{proof}

\begin{corollary} \label{cor:ShiftConst}
        Let $Z \in \pmcrat_d \chquo G (k)$ and $\fct{ChowData}(Z) = (\Gamma , \sigma , \{ Z_v \} )$ be its Chow data. Let $[v,v'] \in E(\Gamma )$ be an edge such that $\sigma (v) \neq 0, \sigma (v') = 0$ and $\val (v) \geq 3$, let $V(\Gamma ) = \Gamma_v \sqcup \Gamma_{v'}$ be the edge decomposition and put
        \[ c_e := \sum_{\substack{i\in[2d]\\p(i)\in\Gamma_v}}c(i)
        - \sum_{\substack{i\in[2d]\\p(i)\in\Gamma_{v'}}}c(i) \neq 0. \]
        For every internal vertex $w$, choose a representative
        \[ x_w=(P_{1,w},\ldots,P_{2d,w},[\ol f_{+,w}:\ol f_{-,w}])
        \quad\text{such that}\quad Z_w=\ol{G\cdot x_w}. \]
        For any $\rho\in k^\times$, there exists a point $Z' \in \pmcrat_d \chquo G (k)$ such that
    \eqns{ \fct{ChowData}(Z' ) &= (\Gamma , \sigma,   \{ Z'_w \}_{w} ), \\ Z'_w &=
    \begin{cases}
    \ol{G\cdot(P_{1,w},\ldots,P_{2d,w},[\rho^{c_e}\ol f_{+,w}:\ol f_{-,w}])}
    & ( w \in \Gamma_{v'}), \\
    Z_w & (w \in \Gamma_v ).
    \end{cases} }
    The cycle in the first case is independent of the chosen representative $x_w$, because scaling the $V_{d+1}$-summand by $\rho^{c_e}$ defines a $G$-equivariant automorphism of $\pmcrat_d$.
\end{corollary}
\begin{proof}
Pick a point $x=(P_1,\ldots,P_{2d},[f_+,f_-]) \in U_X(K)$ such that $\spe \pi_U(x) = Z$, where $U_X$ is the generic degree locus of $\pmcrat_d$. By replacing $x$ by $\gamma_K \cdot x$ for some $\gamma_K \in \pglt (K)$, we assume that $Z_\zeta(x)=Z_v$ and $Z_\xi(x)=Z_{v'}$ for some point $\xi\neq\zeta$ such that $\red_\zeta(\xi)=0$. Set
\[ p=\rho^{-2},\qquad q=1,\qquad s=p/q=\rho^{-2},\qquad
\gamma=\begin{bmatrix}p&0\\0&q\end{bmatrix}, \]
and put $\tau(x):=\ol{\tau_{\gamma,\pm}}(x)$.

The leaves in the selected tangent direction form $\Gamma_{v'}$, so the integers in \Cref{lem:ShiftConstMult} satisfy
\[ a-b=W_e(v'). \]
Since the total leaf-weight is $(d+1)-(d-1)=2$, we have
\[ c_e=W_e(v)-W_e(v')=2-2W_e(v'),
\qquad a-b-1=-\frac{c_e}{2}. \]
Let
\[ g_\rho=\begin{bmatrix}\rho^{-1}&0\\0&\rho\end{bmatrix}\in\slt(k). \]
This element induces $\mu_s$ on $\pone$, and direct calculation on a product of $n$ linear factors gives
\[ \mu_{s,n}(f)=\rho^{-n}(g_\rho\cdot f). \]
Consequently, the marked specialization on the selected branch provided by \Cref{lem:ShiftPoint,lem:ShiftConstMult} is $G$-equivalent to
\eqns{
&(P_{1,w},\ldots,P_{2d,w},
[q p^{a-b-1}\ol f_{+,w}:\ol f_{-,w}])\\
&\qquad=(P_{1,w},\ldots,P_{2d,w},
[\rho^{c_e}\ol f_{+,w}:\ol f_{-,w}]).
}
Here the first equality uses the additional relative factor $s=p/q$ arising from the degree difference between $V_{d+1}$ and $V_{d-1}$, and the second uses $p=\rho^{-2}$ and $a-b-1=-c_e/2$. At the Gauss point and on every other tangent direction, the scalar in \Cref{lem:ShiftConstMult} is $q^{a-b}=1$, and \Cref{lem:ShiftPoint} leaves the marked configurations unchanged.

Therefore, for
\[ Z':=\spe\pi_U(\tau(x)), \]
the strata tree and all cycles on $\Gamma_v$ are unchanged, whereas every cycle on $\Gamma_{v'}$ is modified by the asserted relative factor $\rho^{c_e}$. This proves the result.
\end{proof}

\subsection{Surjectivity and Algebraicity}
\begin{proof}[Proof of \Cref{thm:Ydatum}]
    The construction and injectivity are established in \Cref{prop:YdatumInj}.
    Let $(\Gamma,\sigma,Z)$ be an arbitrary Chow datum. Choose positive integer internal edge lengths and a piecewise-linear function $h$ realizing its contour with the slopes of \cref{eq:weightorder}. Give each leaf edge infinite length. We now construct the marked points as in \Cref{ex:ConfigBerkovich}, read in the reverse direction. Root $\Gamma$ at the leaf $p(1)$, put $P_1=\infty$, and place the adjacent vertex $v_1$ at the Gauss point. For $i\neq1$, put $P_i := \sum_j \alpha_{j,i}t^{d(v_1,v_j)}$, where $j$ runs over the internal vertices on the path from $p(1)$ to $p(i)$. Here $\alpha_{j,i}$ is the coordinate of the direction toward $p(i)$ in the configuration at $v_j$, normalized so that the direction toward the root is $\infty$, and $d$ is the distance for the chosen edge lengths. The common two-dimensional edge orbits in \cref{cond:edgeorbit} ensure that the reductions chosen at the two ends of every internal edge have the same boundary specialization. Choose
    \[ f_+=c_+\prod_{i=1}^{d+1}[P_i,z],\qquad
       f_-=c_-\prod_{i=d+2}^{2d}[P_i,z] \]
    so that the order-difference function is $h$. By \Cref{lem:Orddiff,lem:Xmaxdimorbit}, the resulting specializations have the prescribed signs and reduced orbit supports. It remains only to match the nonzero aspect constants on the zero-contour components, where the aspect constants are shared through edges of $\sigma (v) = \sigma (v') = 0$ by \Cref{cor:boundaryorbit}. 
    Since $k$ is algebraically closed, successive applications of \Cref{cor:ShiftConst} along the rooted tree match the remaining aspect constants while preserving the marked stable curve and the previously matched branches. The resulting generic orbit cycle specializes to $\sum_vZ(v)\in X_d$, proving surjectivity.
\end{proof}
\begin{proof}[Proof of \Cref{cor:stratavariety}]
    Fix a strata tree $(\Gamma , \sigma )$. Then we can construct an injective morphism
    \eqnl{ I : \prod_{v \in V_{\mathrm{st}}(\Gamma)}M_{0,\val (v)} \times \gmult^{\sigma^{-1}(0) \cap V_{\fct{st}}(\Gamma)} \to \fct{Chow}(X, \delta (G,X)) \label{eq:fromstratavariety} }
    by taking the sum of the vertex orbit-cycle families through the addition morphism of the Chow scheme; each cycle corresponding to a vertex satisfies \Cref{def:Chowdatum} \Cref{cond:eachcycle}. The vertex components and their multiplicities are recovered from the irreducible factors of the Chow form, and their marking partitions distinguish the vertices. Since distinct vertex cycles have no common irreducible component, \Cref{lem:ChowFactorization} gives regular recovery of the summands on this decomposition locus. Fixing three distinct marked points on each stable vertex then recovers its configuration and, when present, its aspect parameter by the degree-one slice of \Cref{lem:degreeoneChowSlice}. Consequently $I$ is an embedding onto its locally closed image.
    By \Cref{prop:orbdecomp,cor:boundaryorbit}, the edge-orbit condition \Cref{def:Chowdatum} \Cref{cond:edgeorbit} is automatic unless $e = \{ v, v' \}$ satisfies $\sigma (v) = \sigma (v') = 0$. For such an edge, equality of the two-dimensional boundary orbits has the form $x_v=a_e x_{v'}$, where $a_e$ is an invertible regular function of the configurations. Since the zero-contour subgraph is a forest, these relations leave one free aspect parameter per connected component. By \Cref{thm:Ydatum}, the resulting locus is the tree stratum $(X_d)_{\Gamma,\sigma}$, and hence
    \eqn{ (X_d)_{\Gamma,\sigma}=\fct{Im} I \cap X_d
    \simeq \prod_{v \in V_{\mathrm{st}}(\Gamma)}M_{0,\val (v)} \times \gmult^{\alpha (\Gamma , \sigma)}.}
    Its decomposition stratum is determined only by the multiset of rational-equivalence classes of its irreducible components, so different tree strata may lie in the same decomposition stratum.
\end{proof}

\subsection{Well-Definedness of Multiplier Map}
Put $X=\pmcrat_d$. In either chart of \cref{eq:lFz}, the numerator and denominator vanish simultaneously exactly when $f_-(P)=0$ and the derivative of $f_+$ in a local coordinate at $P$ vanishes. We represent a quotient $A/B$ in that formula by $[B:A]\in\pone$, consistently with the affine coordinate $z_1/z_0$. Since $P_i$ is a root of $f_+$, this gives the following indeterminacy locus of
\eqnl{ \lambda_i : \pmcrat_d  \ratmap  \pone :\  \lambda_i(P_1,\ldots , P_{2d}, F) := \lambda_{F}(P_i) \ ( 1 \leq i \leq d+1) \label{eq:defli} }
:
\eqnl{ I(\lambda_i ) = \{ (P_1,\ldots , P_{2d}, [f_+,f_-]) \mid  \min ( \ord_{P_i}f_+ , \ord_{P_i}f_- + 1) \geq 2 \}. \label{eq:indetli} }

By \Cref{cor:gendegOrbit} and $G$-invariance, $\ol{G\cdot x}\mapsto\lambda_i(x)$ for $x\in U_X$ induces a rational map
\[ \ol{\lambda}_i : X_d  \ratmap \pone . \]

\begin{corollary} \label{cor:markChowQuotMultiplier}
    We have $I(\ol{\lambda}_i ) = \emptyset$, so that $\ol{\lambda_i} : X_d \to \pone$ is a morphism.
\end{corollary}
\begin{proof}
Fix a geometric point $y\in X_d$ with Chow data $\fct{ChowData}(y)=(\Gamma,\sigma,Z)$, and let $v=v_i$ be the internal vertex adjacent to $p(i)$. Then $\val(v)\geq3$: otherwise the split separating $p(i)$ from the remaining leaves has $d(A)=|1-0-1|=0$, giving $\dim G\cdot x_v=2$ by \Cref{prop:orbdecomp}, contrary to \cref{cond:eachcycle}.

Choose $j,k\in[2d]\setminus\{i\}$ in two distinct branches at $v$, different from the branch containing $p(i)$. Fix three distinct points $Q_0,Q_\infty,Q_1\in\pone$, and write
\[ H_r(Q):=\{((P_1,\ldots,P_{2d}),F)\in X\mid P_r=Q\}. \]
For $Z_y:=\sum_w Z(w)$, the unique median $v$ of the three leaves and \Cref{lem:IntersecNum} give that
\[ Z_y\cdot H_i(Q_0)\cdot H_j(Q_\infty)\cdot H_k(Q_1) \]
is an effective zero-cycle of degree one, contributed entirely by $Z(v)$, and represents the unique point of $G\cdot x_v$ with
\[ P_i=Q_0,\qquad P_j=Q_\infty,\qquad P_k=Q_1. \]

By \Cref{lem:degreeoneChowSlice}, after restricting to an open neighborhood $W_{i;jk}$ of $y$, this intersection defines a morphism
\[ \tau_{i;jk}:W_{i;jk}\longrightarrow X. \]
The branch of $p(i)$ at $v$ contains no other marked leaf. Hence, if $f_+\neq0$, then $P_{i,v}$ is a simple root of $f_+$; if $f_+=0$, then $f_-(P_{i,v})\neq0$. In either case $x_v\notin I(\lambda_i)$ by \cref{eq:indetli}. Shrinking $W_{i;jk}$ if necessary, we may assume that $\tau_{i;jk}(W_{i;jk})\cap I(\lambda_i)=\emptyset$.

The compositions
\[ \lambda_i\circ\tau_{i;jk}:W_{i;jk}\longrightarrow\pone \]
are morphisms agreeing with $\ol{\lambda}_i$ on the generic orbit-cycle locus. They therefore glue to $\ol{\lambda}_i:X_d\to\pone$ by irreducibility of $X_d$ and separatedness of $\pone$.
\end{proof}
\begin{remark}
The valuative description gives the same value: for a representative $x_v=((P_{1,v},\ldots,P_{2d,v}),F_v)$ of $Z(v)$,
\[ \ol{\lambda}_i(y)=\lambda_{F_v}(P_{i,v}). \]
\end{remark}

\section{Remarks on Unmarked Chow Quotient}
\label{sec:unmarked}
\subsection{Multiplier Maps on Unmarked Chow Quotient}
For the $\pglt$-equivariant surjective map 
\[ \Pi : \pmcrat_d \to \crat_d : \Pi (P_1,\ldots , P_{2d}; [f_+ : f_-] ) \mapsto [f_+ : f_-], \]
the pushforward of Chow cycles induces the morphism between Chow quotients
\[ \Pi_* : X_d :=  \pmcrat_d \sslash_{Ch} \pglt \to Y_d := \crat_d \sslash_{Ch} \pglt . \]

\begin{proposition} \label{prop:vancycle}
    A three-dimensional reduced orbit closure $Z = \ol{G\cdot x}_{\mathrm{red}}$, where $x = (P_1,\ldots , P_{2d}, [f_+, f_-])\in\pmcrat_d$, satisfies $\Pi_* Z = 0$ if and only if
    \begin{enumerate}
        \item $f_+=0$ and $\# \{ P_{d+2}, \ldots , P_{2d} \} < 3$, or
        \item $f_-=0$ and $\# \{ P_{1}, \ldots , P_{d+1} \} < 3$.
    \end{enumerate}
\end{proposition}
\begin{proof}
    The orbits of dimension at most two in $\crat_d$ are of the form \[  G \cdot [x_0^{d+1} : 0] \text{ or }G \cdot [c_+x_0^{d+1-i}x_1^i : c_-x_0^{d-i}x_1^{i-1}]\  (1 \leq i \leq d),\ (c_+,c_-)\in k^2\setminus\{(0,0)\}. \]
    If $f_+f_- \neq 0$, the inverse image of the orbit is two-dimensional by \Cref{prop:orbdecomp}. Thus we only need to consider the case $f_+ = 0$ or $f_- = 0$, then the assertion is immediate.
    \end{proof}
\begin{lemma} \label{lem:simpleroot}
    For a point $x \in X_d$ and its Chow data $(\Gamma , \sigma , Z)$, let $v_i$ be the internal vertex adjacent to the leaf $p(i)$. We have a bijection between the sets
    \eqns{ & \{ i \in [d+1] \mid \sigma (v_i) \neq + \} \leftrightarrow \\    
    & \left\{ (v, j) \in V(\Gamma ) \times [d+1] : \begin{aligned}
        & \sigma (v) \neq + ,\ \supp Z(v) = \ol{G \cdot (P_{1,v},\ldots , P_{2d,v}; F_v)}_{\mathrm{red}},\\
        & F_v=[f_{+,v}:f_{-,v}],\ \Pi_* Z(v) \neq 0,\ \ord_{P_{j,v}}f_{+,v}<2
    \end{aligned} \right\} \\
    & = \left\{ (v,j) \in V(\Gamma) \times [d+1] :\begin{aligned} &
     \sigma (v) \neq +,\ \Pi_*Z(v) \neq 0,\\ &  \supp Z(v) \setminus I(\lambda_j) \neq \emptyset,\ \lambda_j (Z(v) \setminus I(\lambda_j)) \neq 1 \end{aligned} \right\}. }
    In particular, for the pair $(v,i)$ corresponding to $i$, we have $\ol{\lambda}_i(x) = \lambda_{F_v}(P_{i,v})$.
\end{lemma}
\begin{proof}
    The map from the right to the left is the projection $(v,j)\mapsto j$. To construct its inverse, fix $i$ and start at the internal vertex $v_i$ adjacent to $p(i)$. By \Cref{prop:vancycle}, $\Pi_*Z(v_i)\neq0$ if $\sigma(v_i)=0$. If $\sigma(v_i)=-$ and $\Pi_*Z(v_i)=0$, then there is a unique edge $e=\{v_i,v'_i\}$ pointing toward the other $+$-marked leaves. Move from $v_i$ to $v'_i$ and continue along this path until reaching the first vertex $v$ such that
    \[ \Pi_*Z(v)\neq0\qquad\text{and}\qquad \supp Z(v)\setminus I(\lambda_i)\neq\emptyset. \]
    This construction gives the unique pair $(v,i)$ corresponding to $i$.
\end{proof}
The fixed point multiplier map
\[ \Lambda :  \crat_{d} \ratmap \fct{Sym}_{d+1} \pone \simeq \proj^{d+1} \]
is defined as
\eqns{ \Lambda :  \Pi(U) & \to \fct{Chow}(\pone , (d+1) \cdot \fct{pt}) && \simeq \proj (k [s_0,s_1]_{d+1}) \\ \Lambda ([f_+,f_-]) & \mapsto \sum_{i = 1}^{d+1}[\lambda_i],\quad \lambda_i=[\lambda_{i,0}:\lambda_{i,1}] && \leftrightarrow \prod_{i = 1}^{d+1} (\lambda_{i,0}s_1-\lambda_{i,1}s_0),}
where $U \subset \pmcrat_d$ is the generic degree orbit locus \cref{eq:gendeglocU}, and the affine value of $\lambda_i$ is $\lambda_{i,1}/\lambda_{i,0}$.
This rational map is $\pglt$-invariant and induces the rational map 
\[ \ol{\Lambda} : \crat_d \chquo \pglt \ratmap \proj^{d+1}. \]
Let
\[ \nu_d:Y_d^\nu\longrightarrow Y_d:=\crat_d\sslash_{Ch}G \]
be the normalization. Pullback by $\nu_d$ gives a rational map
\[ \ol{\Lambda}_d^\nu:=\ol{\Lambda}\circ\nu_d:Y_d^\nu\ratmap\proj^{d+1}. \]
\begin{theorem} \label{thm:multwelldef}
    We have $I(\ol{\Lambda}_d^\nu)=\emptyset$. Hence the fixed point multiplier map extends to a morphism
    \[ \ol{\Lambda}_d^\nu:Y_d^\nu\longrightarrow\proj^{d+1}. \]
\end{theorem}
\begin{proof}

    We first attach a point of $\fct{Sym}_{d+1}\pone$ to every geometric Chow point $y\in Y_d$. Choose a geometric lift $x\in\Pi_*^{-1}(y)$ and put
    \[ \Lambda_{\mathrm{cyc}}(y;x):=\sum_{i=1}^{d+1}[\ol{\lambda}_i(x)]. \]
    Let us take the Chow data $\fct{ChowData}(x) = ((\Gamma ,p , c) , \sigma , Z)$ of $x$.
    We remark that if $\sigma (v_i) = +$ for the vertex $v_i$ adjacent to $p(i)$, then we have $\ol{\lambda}_i(x) = 1$ by $\spe f_+ = 0$ and \cref{eq:lFz}.
    By \Cref{lem:simpleroot}, if we put  
    \eqns{ S := \left\{ (v,j) \in V(\Gamma) \times [d+1] :\begin{aligned} &
     \sigma (v) \neq +,\ \Pi_*Z(v) \neq 0,\\ &  \supp Z(v) \setminus I(\lambda_j) \neq \emptyset,\ \lambda_j (Z(v) \setminus I(\lambda_j)) \neq 1 \end{aligned} \right\}, }
    then  we have 
    \eqns{ \Lambda_{\mathrm{cyc}}(y;x) &= \sum_{ (v,j) \in S }[ \lambda_{F_v} (P_{j,v}) ] + \sum_{ \substack{i\in[d+1] :\\  \sigma (v_i) = +}} [1]  \\
    & = (d+1)\cdot [1] + \sum_{(v,j) \in S}  ( [\lambda_{F_v}(P_{j,v}) ] - [1]) }
    as a cycle of $\fct{Chow} (\pone , (d+1)\cdot \fct{pt})$. For a marked representative $(P_1,\ldots,P_{2d};F)$ with $F=[f_+:f_-]$ and $f_+\neq0$, put
    \[ J_F:=\{j\in[d+1]\mid\lambda_F(P_j)\text{ is defined and }\lambda_F(P_j)\neq1\}. \]
    Then
    \[ \sum_{j\in J_F}\bigl([\lambda_F(P_j)]-[1]\bigr)
    =\sum_{\substack{P:\ f_+(P)=0,\\
    \min(\ord_Pf_+,\ord_Pf_-+1)<2}}
    \bigl([\lambda_F(P)]-[1]\bigr). \]
    Each nonzero term comes from a simple root with a unique $+$-marking; admissible multiple roots have multiplier $1$ and contribute zero. The sum therefore depends only on the conjugacy class of $F$.
    For a three-dimensional reduced orbit closure
    \[ C_F:=\ol{G\cdot[F]}_{\mathrm{red}}\subset\crat_d,\qquad F=[f_+:f_-], \]
    define
    \[ \Theta(C_F):=\frac{1}{\#\stab_G[F]}
    \begin{cases}
    \displaystyle\sum_{\substack{P:\ f_+(P)=0,\\
    \min(\ord_Pf_+,\ord_Pf_-+1)<2}}
    \bigl([\lambda_F(P)]-[1]\bigr)&(f_+\neq0),\\[8pt]
    0&(f_+=0),
    \end{cases} \]
    and extend $\Theta$ linearly to rational cycles generated by such orbit closures, with values in rational zero-cycles on $\pone$. This is independent of the representative $F$ of the orbit. For an internal vertex with $\Pi_*Z(v)\neq0$, write $Z(v)=b_v\ol{G\cdot x_v}_{\mathrm{red}}$, where $b_v=\#\stab_G(x_v)$, and put $a_v=\#\stab_G(F_v)$. The restriction of $\Pi$ to the dense orbits has degree $a_v/b_v$, so
    \[ \Pi_*Z(v)=a_v C_{F_v},\qquad
    y=\sum_{\substack{v\in V(\Gamma)\setminus L(\Gamma)\\\Pi_*Z(v)\neq0}}a_v C_{F_v}. \]
    Counting repeated image orbit closures with multiplicity, the preceding calculation gives
    \[ \Lambda_{\mathrm{cyc}}(y;x)=(d+1)[1]+\Theta(y). \]
    This depends only on $y$ and is an effective integral cycle of degree $d+1$ by the marked expression; denote it by $\Lambda_{\mathrm{cyc}}(y)$.

    Choose a dense open subset $V\subset Y_d^\nu$ in the domain of $\ol{\Lambda}_d^\nu$ lying over the generic orbit-cycle locus and covered by the marked generic orbit-cycle locus. Fix a geometric point $\widetilde y\in Y_d^\nu$, put $y=\nu_d(\widetilde y)$, and let $\widetilde y_L$ be any valued lift with generic point in $V$. After a finite extension, choose a marked lift of its generic point and extend it by properness to a valued point $x_L$ of $X_d$. Put $x=\spe x_L$, so $\Pi_*(x)=y$. By \Cref{cor:markChowQuotMultiplier},
    \[ \spe\ol{\Lambda}_d^\nu(\widetilde y_L)
    =\sum_{i=1}^{d+1}[\ol{\lambda}_i(x)]
    =\Lambda_{\mathrm{cyc}}(y). \]
    The specialization is independent of the lift, so normality of $Y_d^\nu$ and \Cref{prop:indetlocus} give the required morphism.
\end{proof}
\begin{remark}
    The morphism $\ol{\Lambda}_d^\nu$ has constant value $\Lambda_{\mathrm{cyc}}(y)$ on each geometric fiber over $y$. Descent to the possibly nonnormal $Y_d$ still requires scheme-theoretic compatibility along the conductor.
\end{remark}

\subsection{Remark: Isomers}

Although the multiplier map is regular on the normalization by \Cref{thm:multwelldef}, the Chow quotient
\[ \crat_d \sslash_{Ch} G \]
does not completely reflect the structure of trees of spheres considered in (\cite{kiwi2015rescaling}, \cite{fujimura-taniguchi2013rational}, \cite{Rumely17}) on the limits of orbits.
\begin{definition}
    Two cycles $Z , Z' \in \pmcrat_d \sslash_{Ch} G$ are \emph{isomers} if $\Pi_* (Z) = \Pi_* (Z')$ and $(\pi_B)_*(Z) = (\pi_B)_*(Z')$.
\end{definition}
\begin{proposition}
    If $d \geq 3$, then there are distinct isomers.
\end{proposition}
\begin{proof}
    Let us take a graph $(\Gamma , p)$ with the following leaf-marking,

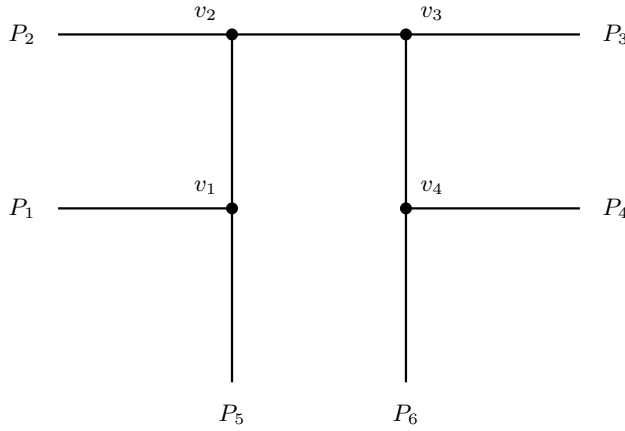
\begin{tikzpicture}[
  x=1.15cm,
  y=1.15cm,
  edge/.style={line width=.85pt},
  vertex/.style={circle,fill=black,inner sep=0pt,minimum size=4.5pt},
  every node/.style={font=\small}
]

\coordinate (v1) at (-1, 1);
\coordinate (v2) at ( 1, 1);
\coordinate (v3) at (-1,-1);
\coordinate (v4) at ( 1,-1);

\draw[edge] (-3,1) -- (v1) -- (v2) -- (3,1);
\draw[edge] (v1) -- (v3) -- (-1,-3);
\draw[edge] (v2) -- (v4) -- ( 1,-3);
\draw[edge] (-3,-1) -- (v3);
\draw[edge] (v4) -- (3,-1);

\node[vertex] at (v1) {};
\node[vertex] at (v2) {};
\node[vertex] at (v3) {};
\node[vertex] at (v4) {};

\node[above left=2pt]  at (v1) {$v_2$};
\node[above right=2pt] at (v2) {$v_3$};
\node[above left=2pt]  at (v3) {$v_1$};
\node[above right=2pt] at (v4) {$v_4$};

\node[left=5pt]  at (-3, 1) {$P_2$};
\node[left=5pt]  at (-3,-1) {$P_1$};
\node[right=5pt] at ( 3, 1) {$P_3$};
\node[right=5pt] at ( 3,-1) {$P_4$};
\node[below=5pt] at (-1,-3) {$P_5$};
\node[below=5pt] at ( 1,-3) {$P_6$};

\end{tikzpicture}

so that \eqns{
    & V(\Gamma ) = \{ l_1 ,\ldots , l_6 , v_1 ,\ldots , v_4 \} , \\
    & E(\Gamma ) = \{ \{ l_i , v_i \}\ (i = 1,2,3,4), \{ l_5, v_1 \}, \{ l_6, v_4 \}, \{ v_j,v_{j+1} \}\ (j = 1,2,3)  \}.
}
Then we have the edge-wise relation $v_1 < v_2 \sim v_3 > v_4$. The contours $ (\sigma (v_1) ,\ldots , \sigma (v_4)) = (+,+,+,0)$ and $(0,+,+,+)$ each give a strata tree for $d = 3$ after inserting zero-sign valence-two vertices wherever the contour crosses zero; these vertices and their orbit cycles are understood to be included. Each corresponding moduli space is $\gmult$, and corresponding points give isomers. For $d > 3$, after renumbering, we may add the remaining $(d-3)$ pairs of $+$ and $-$ leaves at $v_2$, again inserting the required zero-sign valence-two vertices.
\end{proof}
\subsection{Degrees of Orbits}
We can compute both the numerical degree of the orbit cycle $\ol{G \cdot x}$ and that of its reduced support $\ol{G\cdot x}_{\mathrm{red}}$ for a given three-dimensional orbit. A special case of the calculation is the following.
\begin{proposition}
    Let $[f] \in \proj (V_n)$ be a point whose stabilizer is finite. Then the orbit cycle $Z:=\ol{G\cdot[f]}$ satisfies
    \eqnl{ \deg Z=( \# \stab_G([f]) )\deg Z_{\mathrm{red}} = n(n-1)(n-2) - \sum_{P} \delta_n ( \ord_P f ), \label{eq:VnOrbdeg}}
    where $\delta_n (m ) = m(m-1)(3n-2-2m)$.
\end{proposition}
\begin{proof}
    Let $\pi : ( \pone )^n \to \proj (V_n)$ be the morphism such that 
    \[ \pi ([\alpha_{i,0} : \alpha_{i,1}]) = \prod_i ( \alpha_{i,0}x_1 - \alpha_{i,1}x_0) . \]
    Then the morphism $\pi$ is $\slt$-equivariant under the component-wise action on $\pone$. Fix a point $P \in \pi^{-1}([f])$. Then for the orbit cycle $Z' := \ol{G\cdot P}$, $\dim Z' = \dim Z$ and $\pi_* Z' = Z$.
    Write $\{Q_1,\ldots,Q_m\}=\{P_1,\ldots,P_n\}$ for the distinct coordinates of $P$, and put
    \[ a_\ell:=\#\{j\in[n]\mid P_j=Q_\ell\}=\ord_{Q_\ell}f,
    \qquad \ol{a}:=(a_1,\ldots,a_m). \]
    
    Let $H_i$ be the $i$-th coordinate hyperplane divisor of $(\pone )^n$. 
    Then $A^* ( (\pone )^n ) \simeq \zahl [H_1, \ldots , H_n ] / (H_1^2,\ldots , H_n^2 )$ and 
    \[ \eta = \pi^* \oshf_{\proj (V_n)}(1) = \sum_{i = 1}^n H_i. \]
    For the coordinates $ P = (P_1,\ldots , P_n )$, the intersection is computed as 
    \[ H_iH_jH_k \cap Z' = \begin{cases} 1 & ( \# \{ P_i , P_j , P_k \} = 3 ), \\ 0 & ( \# \{ P_i , P_j , P_k \} < 3 ). \end{cases} \]
    For any vector $v = (v_1 ,\ldots , v_m)$, we introduce the notation of power sums and elementary symmetric polynomials as 
\eqnl{ p_k(v) := \sum_{i} v_i^k \text{ and } e_k(v) := \sum_{\substack{ I \subset \{ 1,\ldots , m\}, \\
\# I = k }} \prod_{i \in I} v_i. \label{eq:pses}} Here we remark that 
\[ e_3(v) = \frac{1}{6}(p_1^3(v) - 3 p_1p_2(v) + 2p_3(v)) \]
for any vector $v$.
    \eqns{
     \deg Z & = (c_1(\oshf (1) ) )^3 \cap \pi_* Z' = \pi_* ( \eta^3 \cap Z') \\
     & = \sum_{\# \{ i,j,k \} = 3} a_ia_ja_k \\
     & = 6 e_3( \ol{a}) = p_1(\ol{a})^3 - 3p_1(\ol{a}) p_2(\ol{a}) + 2p_3(\ol{a}) \\
     & = n^3 - \sum_{i} (3na_i^2 - 2a_i^3) \\
     & = n(n-1)(n-2) - \sum_{i} a_i(a_i-1)(3n-2-2a_i).
    }
    This shows the assertion.
\end{proof}
We recall that for $X = \pmcrat_d$, the Chow ring can be written as
\eqns{
A^*(X) 
& = \zahl [H_1,\ldots , H_{2d}, \xi] / (H_i^2 , (\xi - D_+) (\xi - D_-)).
}
So we have 
\eqnsl{ \xi^3 & = \xi \left( \xi (D_+ + D_-) - D_+D_- \right) \nonumber \\ 
& = (D_+ + D_-) ( (D_+ + D_- ) \xi - D_+D_-) - D_+D_- \xi \nonumber \\
& = (D_+^2 + D_+D_- + D_-^2 )\xi - ( D_+^2D_- + D_+ D_-^2 ). \label{eq:xitoD}}

\begin{theorem} \label{thm:degorbit}
    Let $\fct{CG}(F) = (f_+ , f_-)$ be a Clebsch--Gordan form of $F$ and assume $\dim G\cdot[F]=3$, or equivalently that $\stab_G([F])$ is finite. Under the convention $\ord_P0=+\infty$, we have
    \eqnl{ \deg \ol{G\cdot[F]} = \# \stab_G([F]) \cdot \deg \ol{G\cdot[F]}_{\mathrm{red}} = (d+1)d(d-1) - \sum_{P \in \pone} \delta_{d+1} ( c_P), \label{eq:RatdOrbDeg} }
    where $c_P = \min ( \ord_P f_+ , \ord_P f_- + 1)$ and $\delta_{d+1} (c) = c(c-1)(3d+1-2c)$.
\end{theorem}
\begin{proof}
If $f_-=0$, then $c_P=\ord_Pf_+$, and the assertion is exactly \cref{eq:VnOrbdeg} with $n=d+1$. If $f_+=0$, write $b_P=\ord_Pf_-$. Then $c_P=b_P+1$, and \cref{eq:VnOrbdeg} with $n=d-1$ gives the assertion because
\[ \delta_{d+1}(b+1)-\delta_{d-1}(b)=6(d-1)b,
   \qquad \sum_Pb_P=d-1, \]
while
\[ (d^3-d)-(d-1)(d-2)(d-3)=6(d-1)^2. \]
We may therefore assume $f_+f_-\neq0$ for the remainder of the proof.

Let $Z_F:=\ol{G\cdot[F]}$ be the orbit cycle and let $x = (P_1,\ldots , P_{2d}, [f_+,f_-])$ be a point in a fiber of the morphism $\Pi : X \to \ol{\mrat}_d$. Its orbit cycle is
\[ \ol{G\cdot x}=(\#\stab_G(x))\ol{G\cdot x}_{\mathrm{red}}. \]
The equivariant map $\Pi$ has degree $[\stab_G([F]):\stab_G(x)]$ on the orbit. Therefore
\[ \Pi_*\ol{G\cdot x}=Z_F,\qquad \xi^3|_{\ol{G\cdot x}}=\deg Z_F. \]
For the set $S := \{ P_1 ,\ldots , P_{2d} \}$, we write $S$ as $\{ Q_1, \ldots , Q_m \}\ \ ( m := \# S)$ and put
\eqns{ a_i & := \# \{ j \mid P_j = Q_i ,\ 1 \leq j \leq d+1\},\\
b_i & := \# \{ j \mid P_j = Q_i ,\ d+2 \leq j \leq 2d\},\\
c_i & := \min (a_i, b_i+1) }
for $1\leq i\leq m$.
By \Cref{eq:xitoD} and \Cref{lem:IntersecNum}, we have
\eqns{
\xi^3|_{\ol{G\cdot x}}
= & \sum_{i \neq j} ((d+1)-c_i-c_j)(a_ia_j + a_ib_j + b_ib_j)  - \sum_{\# \{ i,j,k \} = 3}(a_ia_jb_k + a_ib_jb_k).
}
The first term on the RHS is given by
\eqns{& \sum_{i \neq j} ((d+1)-c_i-c_j)(a_ia_j + a_ib_j + b_ib_j) \\ 
& = \sum_{i,j} ((d+1)-c_i-c_j)(a_ia_j + \frac{1}{2}a_ib_j + \frac{1}{2}a_jb_i + b_ib_j) \\
& \quad - \sum_i ((d+1)-2c_i)(a_i^2 + a_ib_i + b_i^2) \\
& = (d+1)( (d+1)^2 + (d+1)(d-1) + (d-1)^2) \\
& \ \ - 2 \sum_{i} c_i(a_i(d+1) + a_i\frac{1}{2}(d-1) + b_i\frac{1}{2}(d+1) + b_i(d-1)) \\
& \ \ - \sum_i ((d+1)-2c_i)(a_i^2 + a_ib_i + b_i^2) \\
& = (d+1)(3d^2+1) -  \sum_{ i} c_i((3d+1)a_i + (3d-1) b_i) \\
& \quad - \sum_i ((d+1)-2c_i)(a_i^2 + a_ib_i + b_i^2).
}
In the following computation of the remaining term, we use the notation in \cref{eq:pses} and $\oplus$ as just the concatenation for numerical vectors.
\eqns{
& \sum_{\# \{ i,j,k \} = 3} (a_ia_jb_k + a_ib_jb_k) \\ 
& = 2\left( e_3(\ol{a}\oplus \ol{b}) - e_3(\ol{a}) - e_3(\ol{b}) - \sum_{i \neq j} (a_ia_jb_i + a_ib_ib_j) \right)  \\
& = \frac{1}{3} \Bigl( (p_1^3(\ol{a} \oplus \ol{b}) - 3p_1p_2(\ol{a}\oplus \ol{b}) + 2p_3(\ol{a} \oplus \ol{b})) \\
& \qquad - (p_1^3(\ol{a}) - 3p_1p_2(\ol{a}) + 2p_3(\ol{a})) \\
& \qquad - (p_1^3(\ol{b}) - 3p_1p_2(\ol{b}) + 2p_3(\ol{b})) \Bigr) \\
& \quad - 2\sum_{i,j} (a_ia_jb_i + a_ib_ib_j) + 2\sum_i ( a_i^2b_i + a_ib_i^2) \\
& = \frac{1}{3}\left( (2d)^3 - 6dp_2(\ol{a} \oplus \ol{b}) - ((d+1)^3 - 3(d+1)p_2(\ol{a})) - ((d-1)^3 - 3(d-1)p_2(\ol{b})) \right) \\ 
& \ \ \ \ - 4d \sum_{i} a_ib_i + 2 \sum_i ( a_i^2b_i + a_ib_i^2) \\
& = \frac{1}{3}\left( 6d^3-6d - (3d-3)p_2(\ol{a}) - (3d+3)p_2(\ol{b}) \right) \\
& \quad - 4d \sum_{i} a_ib_i + 2 \sum_i ( a_i^2b_i + a_ib_i^2) \\
& = 2(d^3-d) - (d-1)p_2(\ol{a}) - (d+1)p_2(\ol{b}) - 4d \sum_{i} a_ib_i + 2 \sum_i ( a_i^2b_i + a_ib_i^2).
}
Combining these results, we obtain
\eqns{\xi^3|_{\ol{G\cdot x}} = & (d+1)(3d^2+1) - 2(d+1)(d^2-d) - \sum_{i} c_i((3d+1)a_i + (3d-1) b_i)  \\
& - \sum_i ((d+1)-2c_i)(a_i^2 + a_ib_i + b_i^2) + (d-1)p_2(\ol{a}) + (d+1)p_2(\ol{b}) \\
& + 4d \sum_{i} a_ib_i - 2 \sum_i ( a_i^2b_i + a_ib_i^2) \\
= & (d+1)^3 - \sum_{i} c_i((3d+1)a_i + (3d-1) b_i) + 2\sum_i c_i(a_i^2 + a_ib_i + b_i^2) \\
& + \sum_i \Bigl(-2a_i^2 + (3d-1)a_ib_i - 2(a_i^2b_i + a_ib_i^2)\Bigr). }
Therefore, we have
\[ \deg Z_F = (d+1)^3 - \sum_i q_d(a_i,b_i), \]
where 
    \eqns{ q_d(a,b) & = 2ab(a+b) + 2a^2 - (3d-1)ab \\ 
    & \ \ + \min (a, b+1) \left( (3d+1) a + (3d-1)b  - 2(a^2 + ab + b^2)\right).}
Here we have 
    \eqns{ 
    q_d(a,b) - (3d+1)a  & = \begin{cases}
        a(a-1)(3d+1-2a) & (a \leq b+1) \\
        b(b+1)(3d-1-2b) & (a \geq b+1)
    \end{cases} \\
    & = \delta_{d+1}(\min(a,b+1)).
    }
Then the assertion follows from $\sum_i a_i = d+1$.
\end{proof}
\begin{corollary} \label{cor:autfreeChow}
    The orbit-cycle map extends to an injective morphism
    \[ \iota_d:\fct{rat}_d=\mrat_d/G\longrightarrow Y_d,
    \qquad [F]\longmapsto\ol{G\cdot[F]}
    =\bigl(\#\fct{Aut}(F)\bigr)\ol{G\cdot[F]}_{\mathrm{red}}. \]
    For every $F\in\mrat_d$, the orbit cycle has total degree
    \[ \deg\ol{G\cdot[F]}=d^3-d, \]
    whereas its reduced support has degree $(d^3-d)/\#\fct{Aut}(F)$.
\end{corollary}
\begin{proof}
    A degree-$d\geq2$ rational map has finite automorphism group. By \Cref{cor:OrdDepth}, each $c_P$ is $0$ or $1$ on $\mrat_d$, and hence $\delta_{d+1}(c_P)=0$. The degree assertions follow from \Cref{thm:degorbit}.

    Let $U\subset\mrat_d$ be the locus of maps with trivial automorphism group. The orbit-cycle morphism defining the Chow quotient restricts to a morphism
    \[ \chi_U:U\longrightarrow Y_d. \]
    Let $\Gamma_\chi\subset\mrat_d\times Y_d$ be the closure of its graph and let
    \[ p:\Gamma_\chi\longrightarrow\mrat_d \]
    be the first projection. Closedness of the universal Chow incidence gives $[F]\in\supp Z$ for $([F],Z)\in\Gamma_\chi$, so \Cref{prop:stabilizermultiplicity} yields
    \[ \ol{G\cdot[F]}\leq Z. \]
    Equality of degrees, $\deg \ol{G \cdot [F]}=d^3-d=\deg Z$ by \Cref{thm:degorbit}, gives
    \[ Z=\ol{G\cdot[F]}. \]

    Thus $p$ is finite and birational, hence an isomorphism since $\mrat_d$ is open in $\proj^{2d+1}$ and therefore normal. This extends $\chi_U$ to a $G$-invariant morphism
    \[ \chi:\mrat_d\longrightarrow Y_d,
    \qquad [F]\longmapsto\ol{G\cdot[F]}. \]
    By the universal property of the geometric quotient $\mrat_d\to\fct{rat}_d$ constructed in \cite{silverman1996p1moduli}, this morphism descends to $\iota_d$.

    Injectivity follows because each orbit closure has a unique dense orbit.
\end{proof}
\begin{remark}
    The degrees of generic orbits are given by Kazarnovski\u{i}'s theorem (\cite{Kazar1987formula}; see also \cite{KaverKhovanskii2010KazarThmbyOkounkovBody}). A more unified method for computing orbit degrees in $\proj(V)$ for an arbitrary $\slt$-representation is given in \cite{Deopurkar2026EquivClassOrbCloGL2}.
\end{remark}
\begin{example}
    For the power map $P_d(z)=z^d$, one has $\#\fct{Aut}(P_d)=2(d-1)$. Since there are no holes, all correction terms in \cref{eq:RatdOrbDeg} vanish, and
    \[ \deg\ol{G\cdot[P_d]}_{\mathrm{red}}=\frac{d^3-d}{2(d-1)}=\frac{d(d+1)}2,
    \qquad \deg\ol{G\cdot[P_d]}=d^3-d. \]
    Thus the reduced orbit-closure degrees are $3$ for $d=2$ and $6$ for $d=3$, while the orbit-cycle degrees are $6$ and $24$, respectively.
    
    As a collision check for $d=3$, take the multiplicity pattern $(a_1,b_1)=(2,1)$ and let all remaining roots be distinct and disjoint. Then $c_1=2$, all other $c_i$ are at most $1$, and
    \[ (d^3-d)-\delta_4(2)=24-12=12. \]
    For a generic configuration with this pattern the stabilizer is trivial, so both the reduced orbit-closure degree and the orbit-cycle degree are $12$, in agreement with the expanded formula for $q_3(a_i,b_i)$.
\end{example}

\subsection{Well-Definedness of the Iteration Rational Map}

By \Cref{thm:degorbit}, iteration induces rational maps between the Chow quotient moduli spaces of dynamical systems on the projective line.
\begin{corollary} \label{cor:iterdeg}
    For $n\geq2$, we have
    \[ \delta (G , \Phi_{n} (\ol{\mrat}_d) ) = \delta (G, \ol{\mrat}_{d^n} ). \]
    In particular, the rational map $\Psi_n : Y_d \ratmap Y_{d^n}$ is well-defined.
\end{corollary}
\begin{proof}
For generic $F\in\mrat_d$, the map $F^n\in\mrat_{d^n}$ has finite stabilizer, so \Cref{thm:degorbit} gives orbit-cycle degree
\[ (d^n)^3-d^n=d^{3n}-d^n, \]
equal to that of a generic orbit cycle in $\crat_{d^n}$. As degree determines three-dimensional cycle classes in projective space,
\[ \delta(G,\Phi_n(\ol{\mrat}_d))=\delta(G,\ol{\mrat}_{d^n}). \]
The well-definedness of $\Psi_n$ now follows from \Cref{prop:morphChquo}.
\end{proof}
Combining \cref{cor:iterdeg} with \cref{prop:ChowPushRat}, we obtain the following.
\begin{corollary} \label{cor:iterorbit} 
    Let $n\geq2$. If $(Z,Z')$ is a point of the graph of $\Psi_n$ and $[F] \in \crat_d \setminus I(d)$ satisfies $[F] \in \supp Z$, then we have $\Phi_n ([F]) \in \supp Z'$.
\end{corollary}
\begin{proof}
Choose an irreducible component $C\subset\supp Z$ containing $[F]$. By \Cref{thm:Ydatum} and surjectivity of $\Pi_*$, it has a dense three-dimensional orbit $G\cdot[F_C]$. Since $I(d)$ is closed and $G$-invariant, $[F_C]\notin I(d)$. Hence \Cref{prop:ChowPushRat} gives
\[ \Phi_n([F_C])\in\supp Z'. \]
By equivariance, density of $G\cdot[F_C]$ in $C$, and regularity on $C\setminus I(d)$, the closed set $\supp Z'$ contains $\Phi_n(C\setminus I(d))$, and in particular $\Phi_n([F])$.
\end{proof}

\section{Examples}
\label{sec:examples}
\subsection{The case $d = 2$}
There are two combinatorial types of weighted trees, up to relabeling and suppression of valence-two vertices. Each tree has four leaves $P_1,\ldots,P_4$, colored $+,+,+,-$, respectively.
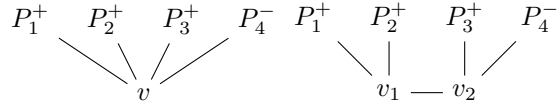
\[ \begin{tikzpicture}
     \node(P1) at (0,0) {$P_1^+$};
     \node(P2) at (1,0) {$P_2^+$};
     \node(P3) at (2,0) {$P_3^+$};
     \node(P4) at (3,0) {$P_4^-$};
     \node(v) at (1.5,-1) {$v$};
     \draw(P1) -- (v);
     \draw(P2) -- (v);
     \draw(P3) -- (v);
     \draw(P4) -- (v);
\end{tikzpicture}
\begin{tikzpicture}
     \node(P1) at (0,0) {$P_1^+$};
     \node(P2) at (1,0) {$P_2^+$};
     \node(P3) at (2,0) {$P_3^+$};
     \node(P4) at (3,0) {$P_4^-$};
     \node(v1) at (1,-1) {$v_1$};
     \node(v2) at (2,-1) {$v_2$};
     \draw(P1) -- (v1);
     \draw(P2) -- (v1);
     \draw(P3) -- (v2);
     \draw(P4) -- (v2);
     \draw(v1) -- (v2);
\end{tikzpicture}
\]
The orderings are given on each edge by counting colorings of leaves on both sides.
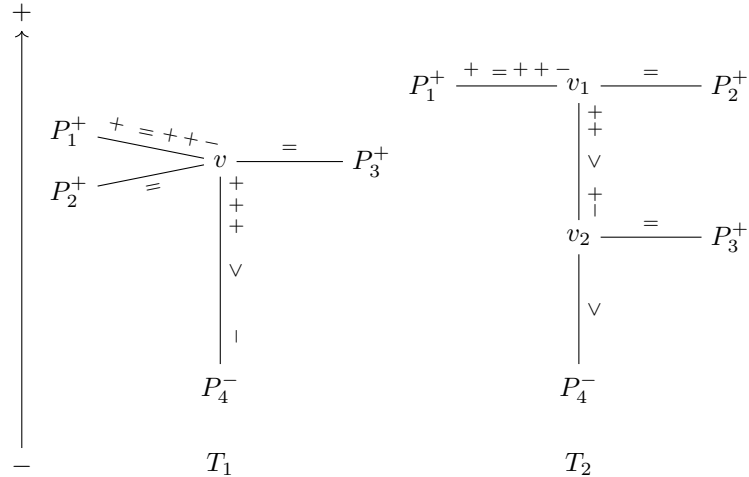
\[
\begin{tikzpicture}
    \node(P) at (0,3) {$+$};
    \node(N) at (0,-3) {$-$};
    \draw[->](N) -- (P);
\end{tikzpicture}
\begin{tikzpicture}
     \node(P1) at (-2,0.4) {$P_1^+$};
     \node(P2) at (-2,-0.4) {$P_2^+$};
     \node(P3) at (2,0) {$P_3^+$};
     \node(P4) at (0,-3) {$P_4^-$};
     \node(v) at (0,0) {$v$};
     \edgecap{P1}{v}{$+$}{$=\ \ $}{$++-$}{above}
     \edgecap{P2}{v}{}{$=$}{}{below}
     \edgecap{P3}{v}{}{$=$}{}{above}
     \edgecap{v}{P4}{$+++$}{$>$}{$-$}{above}
     \node(T1) at (0,-4) {$T_1$};
\end{tikzpicture}
\begin{tikzpicture}
     \node(P1) at (-2,2) {$P_1^+$};
     \node(P2) at (2,2) {$P_2^+$};
     \node(P3) at (2,0) {$P_3^+$};
     \node(P4) at (0,-2) {$P_4^-$};
     \node(v1) at (0,2) {$v_1$};
     \node(v2) at (0,0) {$v_2$};
     \edgecap{P1}{v1}{$+$}{$=\ \ $}{$++-$}{above}
     \edgecap{P2}{v1}{}{$=$}{}{above}
     \edgecap{v1}{v2}{$++$}{$>$}{$+-$}{above}
     \edgecap{P3}{v2}{}{$=$}{}{above}
     \edgecap{v2}{P4}{}{$>$}{}{above}
     \node(T2) at (0,-3) {$T_2$};
\end{tikzpicture}
\]
The zero set of the contour may occur at the indicated points $(z_i)$ and vertices $v,v_i$. This gives three possibilities for $T_1$ and five for $T_2$.
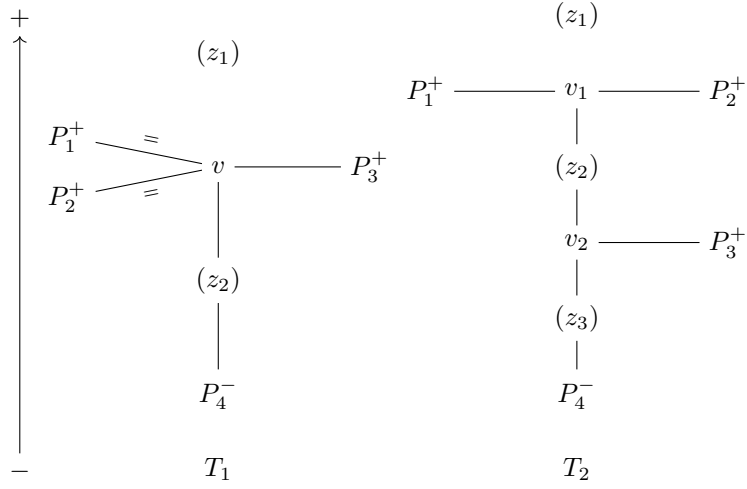
\[
\begin{tikzpicture}
    \node(P) at (0,3) {$+$};
    \node(N) at (0,-3) {$-$};
    \draw[->](N) -- (P);
\end{tikzpicture}
\begin{tikzpicture}
     \node(P1) at (-2,0.4) {$P_1^+$};
     \node(P2) at (-2,-0.4) {$P_2^+$};
     \node(P3) at (2,0) {$P_3^+$};
     \node(P4) at (0,-3) {$P_4^-$};
     \node(v) at (0,0) {$v$};
     \node(z2) at (0,-1.5) {$(z_2)$};
     \node(z1) at (0,1.5) {$(z_1)$};
     \edgecap{P1}{v}{}{$=$}{}{above}
     \edgecap{P2}{v}{}{$=$}{}{below}
     \draw(P3) -- (v);
     \draw(P4) -- (z2);
     \draw(z2) -- (v);
     \node(T1) at (0,-4) {$T_1$};
\end{tikzpicture}
\begin{tikzpicture}
     \node(P1) at (-2,2) {$P_1^+$};
     \node(P2) at (2,2) {$P_2^+$};
     \node(P3) at (2,0) {$P_3^+$};
     \node(P4) at (0,-2) {$P_4^-$};
     \node(v1) at (0,2) {$v_1$};
     \node(v2) at (0,0) {$v_2$};
     \node(z1) at (0,3) {$(z_1)$};
     \node(z2) at (0,1) {$(z_2)$};
     \node(z3) at (0,-1) {$(z_3)$};     
     \draw(P1) -- (v1);
     \draw(P2) -- (v1);
     \draw(P3) -- (v2);
     \draw(P4) -- (z3); \draw(z3) -- (v2);
     \draw(v1) -- (z2);
     \draw(z2) -- (v2);
     \node(T2) at (0,-3) {$T_2$};
\end{tikzpicture}
\]
For each weighted tree with a contour, the associated varieties are
\begin{align*}
\begin{array}{ccc|cc}
    \text{tree} & \text{configuration variety} & \text{\# markings }& \text{contour} & \text{aspect variety} \\
    &  & & (z_1) & \fct{pt} \\
    T_1 & M_{0,4} \simeq \pone \setminus \{ 0, 1,\infty \} & 1 & v & \gmult \simeq \pone \setminus \{ 0,\infty  \} \\
   & & & z_2 & \fct{pt} \\
   \hline 
   & & & (z_1) & \fct{pt} \\
     & & & v_1 & \gmult \\
    T_2 & (M_{0,3} \times M_{0,3}) \simeq \fct{pt} & 3 & z_2 & \fct{pt} \\
     & & & v_2 & \gmult \\
     & & & z_3 & \fct{pt} 
\end{array}
\end{align*}
Put $\rho:=(\pi_B)_*:X_2\to\ol M_{0,4}$ and
\[ X_{2,T_1}:=\rho^{-1}(M_{0,4}). \]
The table parametrizes $X_{2,T_1}$ by $M_{0,4}\times\pone$ and describes two projective lines meeting at one point over each of the three boundary points. We now compute their global gluing.

Normalize the three $+$ markings to $(P_1,P_2,P_3)=(0,1,\infty)$ and write $P_4=t$. On the dense open locus where both summands are nonzero, a representative is
\[ [f_+:f_-]=[x_0x_1(x_1-x_0):u(x_1-tx_0)],
\qquad t\notin\{0,1,\infty\},\quad u\in\gmult. \]
This identifies $\ol M_{0,4}$ with $\pone_t$.
The multiplier formulas \cref{eq:lFz} give
\[ \lambda_1=\frac{tu-2}{tu+1},\qquad
\lambda_2=\frac{(t-1)u+2}{(t-1)u-1},\qquad
\lambda_3=\frac{u+2}{u-1}. \]
By \Cref{cor:markChowQuotMultiplier}, the following are morphisms from $X_2$ to $\pone$:
\[ u=\frac{\ol\lambda_3+2}{\ol\lambda_3-1},\qquad
v=\frac{\ol\lambda_1+2}{1-\ol\lambda_1},\qquad
w=\frac{\ol\lambda_2+2}{\ol\lambda_2-1}. \]
Here the fractions denote M\"obius transformations of the multiplier projective lines, so they are defined also at infinite values. On the dense open locus, $v=tu$ and $w=(t-1)u$.

Let $S$ be the closure of the graph
\[ (t,u)\longmapsto(t,u,tu,(t-1)u)
\quad\text{in }\pone_t\times\pone_u\times\pone_v\times\pone_w. \]
The two product maps have base points
\[ (t,u)=(0,\infty),\ (1,\infty),\ (\infty,0). \]
Near $(0,\infty)$, put $a=1/u$; then $v=t/a$, while $w=(t-1)/a$ is regular as a $\pone$-valued map. Thus the graph is the blowup of $(t,a)$. Near $(1,\infty)$ the same calculation uses $h=t-1$ and $w=h/a$. Near $(\infty,0)$, put $s=1/t$; then
\[ v=u/s,\qquad w=(1-s)v, \]
so the single blowup of $(s,u)$ resolves both products. These centers are disjoint, and away from them both products are morphisms. Consequently
\[ S\simeq\operatorname{Bl}_{\{(0,\infty),(1,\infty),(\infty,0)\}}
       (\pone_t\times\pone_u). \]
For completeness, charts at the nodes of the three boundary fibers are
\[
\begin{array}{c|ccc}
\text{boundary}&h&\alpha&\beta\\ \hline
t=0&t&1/u&v\\
t=1&t-1&1/u&w\\
t=\infty&1/t&1/v&u
\end{array}
\qquad h=\alpha\beta.
\]
Thus each boundary fiber is locally $\alpha\beta=0$. The complementary charts have coordinates $(h,\beta^{-1})$ or $(h,\alpha^{-1})$, with transition relations
\[ \alpha=h\beta^{-1},\qquad \beta=h\alpha^{-1}, \]
and complete the two components to projective lines.

The morphism $(\rho,u,v,w)$ has image $S$, since it is proper and its image contains the dense graph. To check that it identifies the boundary points correctly, write $r=1/\alpha$ and $b=\beta$. The two stable vertices over a boundary point carry the aspect parameters $r$ and $b$, up to fixed nonzero factors. This can be seen directly by changing the dynamical coordinate $z=x_1/x_0$. For $t=h\to0$ and $z=h\zeta$, the transformed pair is
\[ [x_0x_1(hx_1-x_0):hu(x_1-x_0)], \]
so the parameters are $(r,b)=(u,v)$. For $t=1+h\to1$ and $z=1+h\zeta$, it is
\[ [x_0x_1(x_0+hx_1):hu(x_1-x_0)], \]
giving $(r,b)=(u,w)$. Finally, for $t=1/h\to\infty$, inversion $z=1/\zeta$ changes the pair to
\[ [x_0x_1(x_1-x_0):-v(x_1-hx_0)], \]
and then $\zeta=h\eta$ gives
\[ [x_0x_1(hx_1-x_0):-u(x_1-x_0)]. \]
Here $(r,b)=(v,u)$, with a minus sign in each aspect parameter. All displayed pairs are understood up to a common nonzero scalar.

Label the two-plus vertex by $v_1$ and the plus-minus vertex by $v_2$. At $h=0$, the relation $b=hr$ has the following five possibilities, corresponding exactly to the five contours for $T_2$ above:
\[
\begin{array}{c|ccccc}
\text{contour}&(z_1)&v_1&z_2&v_2&z_3\\ \hline
(r,b)&(0,0)&(\gmult,0)&(\infty,0)&(\infty,\gmult)&(\infty,\infty).
\end{array}
\]
The nonzero aspect parameter determines the Chow datum on each of the two open components. At their intersection and at the two outer endpoints there is no remaining aspect parameter. This also accounts for the valence-two zero-contour vertex when the contour lies in an edge or on the ray toward the negative marking. Over $M_{0,4}$, the three possibilities for $T_1$ are likewise distinguished by $u=0$, $u\in\gmult$, and $u=\infty$.

It follows from \Cref{thm:Ydatum} that $(\rho,u,v,w):X_2\to S$ is bijective on geometric points. Being proper, it gives the claimed homeomorphism with the three-point blowup. For $k=\cpx$, it is also a homeomorphism for the complex analytic topology, since the source is compact and the target is Hausdorff.

\bibliographystyle{amsalpha}
\bibliography{refs}

@misc{OpenAI2026ChatGPT,
  author       = {{OpenAI}},
  title        = {{ChatGPT} ({GPT-5.6})},
  year         = {2026},
  howpublished = {\url{https://chatgpt.com/}},
  note         = {AI assistant; accessed September 6, 2026}
}

@misc{OpenAI2026Prism,
  author       = {{OpenAI}},
  title        = {{Prism}},
  year         = {2026},
  howpublished = {\url{https://openai.com/prism/}},
  note         = {AI-assisted research and {\LaTeX} writing workspace; accessed September 6, 2026}
}

@misc{stacks-project,
  author       = {The {Stacks project authors}},
  title        = {{The Stacks Project}},
  howpublished = {\url{https://stacks.math.columbia.edu}},
  year         = {2026}
}

@incollection{DalbecSturmfels1995ChowForms,
  author    = {Dalbec, John and Sturmfels, Bernd},
  title     = {Introduction to {Chow} Forms},
  booktitle = {Invariant Methods in Discrete and Computational Geometry},
  editor    = {White, Neil L.},
  publisher = {Springer},
  address   = {Dordrecht},
  year      = {1995},
  pages     = {37--58},
  doi       = {10.1007/978-94-015-8402-9_2}
}

@phdthesis{west2015moduli,
  author = {West, Lloyd William},
  title = {The moduli space of rational maps},
  school = {City University of New York},
  year = {2015},
  pages = 88,
  isbn = {978-1339-02142-3}
}

@article{silverman1996p1moduli,
  title={The space of rational maps on {$\mathbb{P}^1$}},
  author={Silverman, Joseph H},
  journal={Duke Mathematical Journal},
  volume={94},
  number={1},
  year={1998},
  pages={41--77},
  publisher={Duke University Press}
}

@article{schmitt2017cptf_stmaps,
  title={A compactification of the moduli space of self-maps of $\mathbb{CP}^1$ via stable maps},
  author={Schmitt, Johannes},
  journal={Conformal Geometry and Dynamics of the American Mathematical Society},
  volume={21},
  number={11},
  pages={273--318},
  year={2017}
}

@book{gkz1994disc-res-mult,
    author = {Gel'fand, I. M. and Kapranov, M. M. and Zelevinsky, A. V.},
    title = {Discriminants, resultants, and multidimensional determinants},
    series={Mathematics: Theory \& Applications},
    year={1994},
    publisher={Birkhäuser Boston, Inc., Boston, MA},
    pages={x+523},
    isbn={0-8176-3660-9}
}

@article {kapranov1993VeroneseAndM0n,
    AUTHOR = {Kapranov, M. M.},
     TITLE = {Veronese curves and {G}rothendieck-{K}nudsen moduli space
              {$\overline M_{0,n}$}},
   JOURNAL = {J. Algebraic Geom.},
  FJOURNAL = {Journal of Algebraic Geometry},
    VOLUME = {2},
      YEAR = {1993},
    NUMBER = {2},
     PAGES = {239--262},
      ISSN = {1056-3911,1534-7486},
   MRCLASS = {14H10 (14C05 14D99)},
  MRNUMBER = {1203685},
MRREVIEWER = {R.\ F.\ Lax},
}

@InCollection{kapranov1993ChowQuotGrassI,
 Author = {Kapranov, M. M.},
 Title = {Chow quotients of {Grassmannians}. {I}},
 BookTitle = {I. M. Gelfand seminar. Part 2: Papers of the Gelfand seminar in functional analysis held at Moscow University, Russia, September 1993},
 ISBN = {0-8218-4119-X},
 Pages = {29--110},
 Year = {1993},
 Publisher = {Providence, RI: American Mathematical Society},
 Language = {English},
 zbMATH = {475119},
 Zbl = {0811.14043}
}

@article{fujimura-taniguchi2013rational,
  title={Rational functions with nodes},
  author={Fujimura, Masayo and Taniguchi, Masahiko},
  year={2013},
  journal={Journal of Analysis},
  volume = {21},
  pages = {85--100},
}

@article {kiwi2015rescaling,
    AUTHOR = {Kiwi, Jan},
     TITLE = {Rescaling limits of complex rational maps},
   JOURNAL = {Duke Math. J.},
  FJOURNAL = {Duke Mathematical Journal},
    VOLUME = {164},
      YEAR = {2015},
    NUMBER = {7},
     PAGES = {1437--1470},
      ISSN = {0012-7094,1547-7398},
   MRCLASS = {37F45 (12J25 26E30 32G15 32H50 37P20 37P50)},
  MRNUMBER = {3347319},
MRREVIEWER = {Romain\ Dujardin},
       DOI = {10.1215/00127094-2916431},
       URL = {https://doi.org/10.1215/00127094-2916431},
}

@ARTICLE{kiwi-nie2023indet-loci,
    author = "Jan Kiwi and Hongming Nie",
     title = "Indeterminacy loci of iterate maps in moduli space",
   journal = "Indiana Univ. Math. J.",
  fjournal = "Indiana University Mathematics Journal",
    volume = 72,
      year = 2023,
     number = 3,
     pages = "969--1026",
      issn = "0022-2518",
     coden = "IUMJAB",
   mrclass = "",
}

@Article{YiHu2005ChowQuotTop,
 Author = {Hu, Yi},
 Title = {Topological aspects of {Chow} quotients},
 FJournal = {Journal of Differential Geometry},
 Journal = {J. Differ. Geom.},
 ISSN = {0022-040X},
 Volume = {69},
 Number = {3},
 Pages = {399--440},
 Year = {2005},
 Language = {English},
 DOI = {10.4310/jdg/1122493996},
 zbMATH = {5004287},
 Zbl = {1087.14032}
}

@article {Kazar1987formula,
    AUTHOR = {Kazarnovski\u{i}, B. Ya.},
     TITLE = {Newton polyhedra and {B}ezout's formula for matrix functions
              of finite-dimensional representations},
   JOURNAL = {Funktsional. Anal. i Prilozhen.},
  FJOURNAL = {Akademiya Nauk SSSR. Funktsional\cprime ny\u i\ Analiz i ego
              Prilozheniya},
    VOLUME = {21},
      YEAR = {1987},
    NUMBER = {4},
     PAGES = {73--74},
      ISSN = {0374-1990},
   MRCLASS = {22E45 (14L32)},
  MRNUMBER = {925078},
MRREVIEWER = {V.\ L.\ Popov},
}

@article {KaverKhovanskii2010KazarThmbyOkounkovBody,
    AUTHOR = {Kaveh, Kiumars and Khovanskii, A. G.},
     TITLE = {Moment polytopes, semigroup of representations and
              {K}azarnovskii's theorem},
   JOURNAL = {J. Fixed Point Theory Appl.},
  FJOURNAL = {Journal of Fixed Point Theory and Applications},
    VOLUME = {7},
      YEAR = {2010},
    NUMBER = {2},
     PAGES = {401--417},
      ISSN = {1661-7738,1661-7746},
   MRCLASS = {20G05 (05E10)},
  MRNUMBER = {2729398},
MRREVIEWER = {Anthony\ Henderson},
       DOI = {10.1007/s11784-010-0027-7},
       URL = {https://doi.org/10.1007/s11784-010-0027-7},
}

@book {Kollar1996RatCurvesOnAlgVar,
    AUTHOR = {Koll\'ar, J\'anos},
     TITLE = {Rational curves on algebraic varieties},
    SERIES = {Ergebnisse der Mathematik und ihrer Grenzgebiete. 3. Folge. A
              Series of Modern Surveys in Mathematics [Results in
              Mathematics and Related Areas. 3rd Series. A Series of Modern
              Surveys in Mathematics]},
    VOLUME = {32},
 PUBLISHER = {Springer-Verlag, Berlin},
      YEAR = {1996},
     PAGES = {viii+320},
      ISBN = {3-540-60168-6},
   MRCLASS = {14-02 (14C05 14E05 14F17 14J45)},
  MRNUMBER = {1440180},
MRREVIEWER = {Yuri\ G.\ Prokhorov},
       DOI = {10.1007/978-3-662-03276-3},
       URL = {https://doi.org/10.1007/978-3-662-03276-3},
}

@article{Rumely17,
	author = {Rumely, Robert},
	doi = {10.2140/ant.2017.11.841},
	fjournal = {Algebra \& Number Theory},
	issn = {1937-0652},
	journal = {Algebra Number Theory},
	mrclass = {37P50 (11S82 37P05)},
	mrnumber = {3665639},
	mrreviewer = {Liang-Chung Hsia},
	number = {4},
	pages = {841--884},
	title = {A new equivariant in nonarchimedean dynamics},
	url = {https://doi.org/10.2140/ant.2017.11.841},
	volume = {11},
	year = {2017}}

@book {BakerRumely2010PotTheoryDynOnBerkoP1,
    AUTHOR = {Baker, Matthew and Rumely, Robert},
     TITLE = {Potential theory and dynamics on the {B}erkovich projective
              line},
    SERIES = {Mathematical Surveys and Monographs},
    VOLUME = {159},
 PUBLISHER = {American Mathematical Society, Providence, RI},
      YEAR = {2010},
     PAGES = {xxxiv+428},
      ISBN = {978-0-8218-4924-8},
   MRCLASS = {37P50 (14G20 14G22 31C15 31C45 37P40)},
  MRNUMBER = {2599526},
MRREVIEWER = {Charles\ Favre},
       DOI = {10.1090/surv/159},
       URL = {https://doi.org/10.1090/surv/159},
}

@article {Arfeux2017DynTreeSphere,
    AUTHOR = {Arfeux, Matthieu},
     TITLE = {Dynamics on trees of spheres},
   JOURNAL = {J. Lond. Math. Soc. (2)},
  FJOURNAL = {Journal of the London Mathematical Society. Second Series},
    VOLUME = {95},
      YEAR = {2017},
    NUMBER = {1},
     PAGES = {177--202},
      ISSN = {0024-6107,1469-7750},
   MRCLASS = {37F20 (30D05 30F60 32G15)},
  MRNUMBER = {3653089},
MRREVIEWER = {Kevin\ M.\ Pilgrim},
       DOI = {10.1112/jlms.12016},
       URL = {https://doi.org/10.1112/jlms.12016},
}

@article {DeMarcoFaber2014DegenCpxDynSys,
    AUTHOR = {DeMarco, Laura and Faber, Xander},
     TITLE = {Degenerations of complex dynamical systems},
   JOURNAL = {Forum Math. Sigma},
  FJOURNAL = {Forum of Mathematics. Sigma},
    VOLUME = {2},
      YEAR = {2014},
     PAGES = {Paper No. e6, 36},
      ISSN = {2050-5094},
   MRCLASS = {37F10 (37F45 37P50)},
  MRNUMBER = {3264250},
MRREVIEWER = {Xavier\ Jarque},
       DOI = {10.1017/fms.2014.8},
       URL = {https://doi.org/10.1017/fms.2014.8},
}

@article {Arfeux2017CptfTreesOfSpheresCovers,
    AUTHOR = {Arfeux, Matthieu},
     TITLE = {Compactification and trees of spheres covers},
   JOURNAL = {Conform. Geom. Dyn.},
  FJOURNAL = {Conformal Geometry and Dynamics. An Electronic Journal of the
              American Mathematical Society},
    VOLUME = {21},
      YEAR = {2017},
     PAGES = {225--246},
      ISSN = {1088-4173},
   MRCLASS = {37F20 (30D05 30F60 32G15)},
  MRNUMBER = {3645509},
MRREVIEWER = {Kevin\ M.\ Pilgrim},
       DOI = {10.1090/ecgd/309},
       URL = {https://doi.org/10.1090/ecgd/309},
}

@article {DoyleSilverman2020ModuliSpDynSysPortraits,
    AUTHOR = {Doyle, John R. and Silverman, Joseph H.},
     TITLE = {Moduli spaces for dynamical systems with portraits},
   JOURNAL = {Illinois J. Math.},
  FJOURNAL = {Illinois Journal of Mathematics},
    VOLUME = {64},
      YEAR = {2020},
    NUMBER = {3},
     PAGES = {375--465},
      ISSN = {0019-2082,1945-6581},
   MRCLASS = {37P45 (37P15)},
  MRNUMBER = {4132597},
MRREVIEWER = {Christian\ Lehn},
       DOI = {10.1215/00192082-8642523},
       URL = {https://doi.org/10.1215/00192082-8642523},
}

@article {DeMarco2005IterAtBoundary,
    AUTHOR = {DeMarco, Laura},
     TITLE = {Iteration at the boundary of the space of rational maps},
   JOURNAL = {Duke Math. J.},
  FJOURNAL = {Duke Mathematical Journal},
    VOLUME = {130},
      YEAR = {2005},
    NUMBER = {1},
     PAGES = {169--197},
      ISSN = {0012-7094,1547-7398},
   MRCLASS = {37F45 (32H50 37F10 37F35 37F40)},
  MRNUMBER = {2176550},
MRREVIEWER = {Kevin\ M.\ Pilgrim},
       DOI = {10.1215/S0012-7094-05-13015-0},
       URL = {https://doi.org/10.1215/S0012-7094-05-13015-0},
}

@article{FavreGong2025non,
  title={Non-Archimedean Techniques and Dynamical Degenerations},
  author={Favre, Charles and Gong, Chen},
  journal={Peking Mathematical Journal},
  pages={1--63},
  year={2025},
  publisher={Springer},
  doi={10.1007/s42543-025-00100-7}
}

@misc{Roy2025ActionGroupeCptfHybarXiv,
      title={Action de groupe alg\'ebrique sur la compactification hybride},
      author={Alexandre Roy},
      year={2025},
      eprint={2512.00201},
      archivePrefix={arXiv},
      primaryClass={math.AG},
      note={arXiv:2512.00201v2 [math.AG], 19 December 2025},
      url={https://arxiv.org/abs/2512.00201},
}

@article {Deopurkar2026EquivClassOrbCloGL2,
    AUTHOR = {Deopurkar, Anand},
     TITLE = {Equivariant classes of orbit closures in {${\rm
              GL}(2)$}-representations},
   JOURNAL = {Selecta Math. (N.S.)},
  FJOURNAL = {Selecta Mathematica. New Series},
    VOLUME = {32},
      YEAR = {2026},
    NUMBER = {3},
     PAGES = {Paper No. 60, 40},
      ISSN = {1022-1824,1420-9020},
   MRCLASS = {14C15 (14N10)},
  MRNUMBER = {5089005},
       DOI = {10.1007/s00029-026-01166-6},
       URL = {https://doi.org/10.1007/s00029-026-01166-6},
}

\end{document}